%% file: real-etale.tex
\documentclass{amsart}
\usepackage[margin=1in]{geometry}
\input{preamble.tex}

\usepackage{mathrsfs}
\usepackage{tikz-cd}
\usepackage{verbatim}

\title{Motivic and Real Étale Stable Homotopy Theory}

\author{Tom Bachmann}
\email{tom.bachmann@zoho.com}
\address{Fakultät Mathematik \\ Universität Duisburg-Essen \\%
              Thea-Leymann-Straße 9 \\ 45127 Essen \\ Germany}

\newcommand{\sSH}{\mathrm{SH}}

\begin{document}

\begin{abstract}
Let $S$ be a Noetherian scheme of finite dimension
and denote by $\rho \in [\tunit, \Gm]_{\SH(S)}$ the (additive inverse of the)
morphism corresponding to $-1 \in \mathcal{O}^\times(S)$. Here $\SH(S)$ denotes
the motivic stable homotopy category. We show that the category obtained by
inverting $\rho$ in $\SH(S)$
is canonically equivalent to the (simplicial) local stable homotopy
category of the site $S_\ret$, by which we mean the \emph{small} real étale site
of $S$, comprised of étale schemes over $S$ with the real étale topology.

One immediate application is that $\SH(\RR)[\rho^{-1}]$ is equivalent to the
classical stable homotopy category. In particular this computes all the stable
homotopy sheaves of the $\rho$-local sphere (over $\RR$). As further
applications we show that $D_\Aone(k, \ZZ[1/2])^- \wequi \DM_W(k)[1/2]$
(improving a result of Ananyevskiy-Levine-Panin), reprove Röndigs' result that
$\ul{\pi}_i(\tunit[1/\eta,1/2]) = 0$ for $i =1, 2$ and establish some new rigidity
results.
\end{abstract}

\maketitle

\tableofcontents


\section{Introduction}

For a scheme $S$ we denote by $\SH(S)$ the motivic stable homotopy category
\cite{A1-homotopy-theory,ayoub2007six}. We recall that this is a
triangulated category which is the homotopy category of a stable model category
that (roughly) is obtained from the homotopy theory of (smooth, pointed) schemes by
making the ``Riemann sphere'' $\mathbb{P}^1_S$ into an invertible object.

If $\alpha: k \hookrightarrow \CC$ is an embedding of a field $k$ into the complex
numbers, then we obtain a complex realisation functor $R_{\alpha,\CC}: \SH(k)
\to \SH$ (where now $\SH$ denotes the classical stable homotopy category)
connecting the world of motivic stable homotopy theory to classical
stable homotopy theory \cite[Section 3.3.2]{A1-homotopy-theory}. This functor is induced from the functor
which sends a smooth scheme $S$ over $k$ to its topological space of
of complex points $S(\CC)$ (this depends on $\alpha$). Similarly if $\beta: k
\hookrightarrow \RR$ is an embedding into the real numbers, then there is a real
realisation functor $R_{\beta,\RR}: \SH(k) \to \SH$ induced from $S \mapsto
S(\RR)$ \cite[Section 3.3.3]{A1-homotopy-theory} \cite[Proposition
4.8]{heller2016galois}.

These functors serve as a good source of inspiration and a convenient test of
conjectures in stable motivic homotopy theory. For example, in order for a
morphism $f: E \to F$ to be an equivalence it is necessary that
$R_{\alpha,\CC}(f)$ and $R_{\beta,\RR}(f)$ are equivalences, for all such
embeddings $\alpha, \beta$. On the other hand, this criterion is clearly not
sufficient---there are fields without any real or complex embeddings!

It is thus a very natural question to ask how far these functors are from being
an equivalence, or what their ``kernel'' is. The aim of this article is to give
some kind of complete answer to this question in the case of real realisation.
We begin with the simplest formulation of our result. Write $R_\RR$ for the
(unique) real realisation functor for the field $k=\RR$. The first clue comes from
the observation that $R_\RR(\Gm) = \RR \setminus 0 \wequi \{\pm 1\} = S^0$. That
is to say $R_\RR$ identifies $\Gm$ and $S^0$. We can even do better. Write
$\rho': S^0 \to \Gm$ for the map of pointed motivic spaces corresponding to $-1
\in \RR^\times$. Then one may check easily that $R_\RR(\rho')$ is an equivalence
between $S^0 \wequi R_\RR(S^0)$ and $R_\RR(\Gm)$.

We prove that $\SH(\RR)[\rho'^{-1}] \wequi \SH$ via real
realisation. That is to say $R_\RR$ is in some sense the universal functor
turning $\rho'$ into an equivalence.
More precisely, the functor $R_\RR: \SH(\RR) \to \SH$ has a right
adjoint $R^*$ (e.g. by Neeman's version of Brown representability) and we show
that $R^*$ is fully faithful with image consisting of the $\rho'$-stable motivic
spectra, i.e. those $E \in \SH(\RR)$ such that $E(X\wedge \Gm)
\xrightarrow{\rho^*} E(X)$ is an equivalence for all $X \in Sm(\RR)$.

Of course, our description of $\SH(\RR)[\rho'^{-1}]$ is just an explicit
description of a certain Bousfield localisation of $\SH(\RR)$. Moreover the
element $\rho'$ exists not only over $\RR$ but already over $\ZZ$, so we are
lead to study more generally the category $\SH(S)[\rho'^{-1}]$, for more or less
arbitrary base schemes $S$. Actually, for some formulas it is nicer to consider
$\rho := -\rho' \in [S, \Sigma^\infty \Gm]$ and we shall write this from now on.
Of course $\SH(S)[\rho'^{-1}] = \SH(S)[\rho^{-1}]$. In this generality we can of
longer expect that $\SH(S)[\rho^{-1}] \wequi \SH$. Indeed as we have
said before in general there is no real realisation! As a first attempt, one
might guess that if $X$ is a scheme over $\RR$, then $\SH(S)[\rho^{-1}] \wequi
\SH(S(\RR))$, where the right hand side denotes some form of parametrised
homotopy theory \cite{may2006parametrized}. This cannot be quite true unless $S$
is proper, because the category $\SH(S(\RR))$ will then not be compactly
generated. The way out is to use semi-algebraic topology.
For this we have to recall that if $S$ is a scheme, then there
exists a topological space $R(S)$ \cite[(0.4.2)]{real-and-etale-cohomology}. Its points are pairs $(x,
\alpha)$ with $x \in S$ and $\alpha$ an ordering of the residue field $k(x)$.
This is given a topology incorporating all of these orderings. Write $Shv(RS)$
for the category of sheaves on this topological space.

Now, given any topos $\mathcal{X}$, there is a naturally associated stable
homotopy category $\sSH(\mathcal{X})$. If $\mathcal{X} \wequi Set$ then
$\sSH(\mathcal{X})$ is just the ordinary stable homotopy category. In general, if
$\mathcal{X} \wequi Shv(\mathcal{C})$ where $\mathcal{C}$ is a Grothendieck
site, then $\sSH(X)$ is the local homotopy category of presheaves of spectra on
$\mathcal{C}$.

With this preparation out of the way, we can state our main result:

\begin{thm*}[(see Theorem \ref{thm:final-comparison})]
Let $S$ be a Noetherian scheme of finite dimension. Then there is a canonical
equivalence of categories
\[ \SH(S)[\rho^{-1}] \wequi \sSH(Shv(RS)). \]
\end{thm*}

A more detailed formulation is given later in this introduction. For now let us
mention one application. We go back to $S = Spec(\RR)$. In this case Proposition
\ref{prop:real-realn-compatible} in Section \ref{sec:real-realisation}
assures us that the equivalence from the above
theorem does indeed come from real realisation. But given $E \in \SH(\RR)$, its
$\rho$-localisation can be calculated quite explicitly (see Lemma
\ref{lemm:rho-local-model}). From this one concludes that $\pi_i(R_\RR E) =
\colim_n \ul{\pi}_i(E)_n(\RR)$, where the colimit is along multiplication by
$\rho$ in the second grading of the bigraded homotopy sheaves of $E$. (Recall
that $\ul{\pi}_i(E)_n(\RR) = [\tunit[i], E \wedge \Gmp{n}]$ so $\rho$ indeed induces
$\rho: \ul{\pi}_i(E)_n(\RR) \to \ul{\pi}_i(E)_{n+1}(\RR)$.)

This may seem slightly esoteric, but
actually $\SH(S)[\rho^{-1}, 2^{-1}] = \SH(S)[\eta^{-1}, 2^{-1}]$ and so our
computations apply, after inverting two, to the more conventional
$\eta$-localisation as well. As a corollary, we obtain the following.

\begin{thm*}
The motivic stable 2-local, $\eta$-local stems over $\RR$ agree with the classical
stable 2-local stems:
\[ \ul{\pi}_i(\tunit_{\eta, 2})_j(\RR) = \pi_i^s \otimes_\ZZ \ZZ[1/2]. \]
\end{thm*}

Some more applications will be described later in this introduction.

\paragraph{Overview of the proof.}
The proof uses a different description of the category $Shv(RS)$. Namely, there
is a topology on all schemes called the \emph{real étale topology} and
abbreviated rét-topology \cite[(1.2)]{real-and-etale-cohomology}. (The covers are families of étale
morphisms which induce a jointly surjective family on the associated real spaces $R(\bullet)$.)
We write $Sm(S)_\ret$ for the site of all smooth
schemes over $S$ with this topology, and $S_\ret$ for the site of all étale
schemes over $S$ with this topology. Then $Shv(S_\ret) \wequi Shv(RS)$
\cite[Theorem (1.3)]{real-and-etale-cohomology}.

Write $\SH(S)$ for the motivic stable homotopy
category, $\SH(S)[\rho^{-1}]$ for the $\rho$-local motivic stable homotopy
category, $\SH(S)^\ret$ for the rét-local motivic stable homotopy category (i.e.
the category obtained from the site $Sm(S)_\ret$ by precisely the same
construction as is used to build $\SH(S)$ from $Sm(S)_{Nis}$), and
$\SH^{S^1}(S)$ for the motivic $S^1$-stable homotopy category. We trust that
$\SH^{S^1}(S)^\ret$, $\SH(S)^\ret[\rho^{-1}]$ and so on have evident meanings.
Write $\sSH(S_\ret)$ for the rét-local stable homotopy category on the small real
étale site. This is just the homotopy category of the category of presheaves of spectra on $S_\ret$ with
the local model structure. Similarly $\sSH(Sm(S)_\ret)$ means the rét-local
presheaves of spectra on $Sm(S)$. Then for example $\SH^{S^1}(S)^\ret$ is the
$\Aone$-localisation of $\sSH(Sm(S)_\ret)$.

The canonical functor $e: \sSH(S_\ret) \to \sSH(Sm(S)_\ret)$ (extending a (pre)sheaf
on the small site to the large site) is fully faithful by general results (see
Corollary \ref{corr:Le-exact-ff}). It is moreover $t$-exact: for $E \in
\sSH(S_\ret)$ we have $\ul{\pi}_i(eE) = e\ul{\pi}_i(E)$. Here $\ul{\pi}_*$
denotes the homotopy sheaves.

If $F$ is
a sheaf on the small real étale site of a scheme $Y$, then $H^p(Y \times \Aone,
F) = H^p(Y, F)$ and $H^p(Y_+ \wedge \Gm, F) = H^p(Y, F)$. If $Y$ is of finite
type over $\RR$ and $F$ is locally constant,
then this follows by comparison of real étale cohomology with
Betti cohomology of the real points \cite[Theorem II.5.7]{delfs1991homology}. For
the general case, see Theorem \ref{thm:ret-htpy-rho-inv}.

Now the category $\SH^{S^1}(S)^\ret[\rho^{-1}]$
is obtained from $\sSH(Sm(S)_\ret)$ by $(\Aone, \rho)$-localisation. It follows
from $t$-exactness of $e$, the descent spectral sequence, and the above result
about rét-cohomology that the composite $\sSH(S_\ret) \to \sSH(Sm(S)_\ret) \to
\SH^{S^1}(S)^\ret[\rho^{-1}]$ is still fully faithful.

The category $\SH(S)^\ret[\rho^{-1}]$ is obtained from
$\SH^{S^1}(S)^\ret[\rho^{-1}]$ by $\otimes$-inverting $\Gm$. However in the
latter category we have $\Gm \wequi \tunit$ (via $\rho$!), so $\Gm$ is already
invertible, and inverting it has no effect: $\SH^{S^1}(S)^\ret[\rho^{-1}]
\wequi \SH(S)^\ret[\rho^{-1}]$. We have thus shown that
\[ \sSH(S_\ret) \to \SH(S)^\ret[\rho^{-1}] \]
is fully faithful.

The next step is to show that it is essentially surjective. This follows from
the proper base change theorem by a clever argument of Cisinski-Déglise. Of
course this first requires that we know that $\sSH(S_\ret)$ and
$\SH(S)^\ret[\rho^{-1}]$ satisfy proper base change. For $\sSH(S_\ret)$ this is
a consequence of the proper base change theorem in real étale cohomology
established by Scheiderer, see Theorem \ref{thm:ret-proper-basechange}. For
$\SH(S)^\ret[\rho^{-1}]$ this would follow from the axiomatic six functors
formalism of Voevodsky/Ayoub/Cisinski-Déglise, see Section
\ref{sec:recoll-pre-motivic}. It is in fact not very hard to show directly that
$\SH(S)^\ret[\rho^{-1}]$ satisfies the six functors formalism. Instead we shall
show (without assuming the six functors formalism) that $\SH(S)^\ret[\rho^{-1}]
\wequi \SH(S)[\rho^{-1}]$, and that this latter category satisfies the six
functors formalism.

The next step is thus to show that the localisation functor $\SH(S)[\rho^{-1}] \to
\SH(S)^\ret[\rho^{-1}]$ is an equivalence. It clearly has dense image, so it
suffices to show that it is fully faithful. Using the fact that
$\SH(S)[\rho^{-1}]$ satisfies continuity and gluing (which follows quite easily
from the same statement for $\SH(S)$), we may reduce to the case where $S$ is the
spectrum of a field $k$. The case where $char(k) > 0$ is easily dealt with (note
that such fields are never orderable), so we may assume that $k$ has
characteristic zero and so in particular is \emph{perfect}.

The $\rho$-localisation can be described rather explicitly. For $E \in
\SH(k)$, consider the directed system
\[ E \xrightarrow{\rho} E \wedge \Gm \xrightarrow{\rho} E \wedge \Gm \wedge
     \Gm \xrightarrow{\rho} \dots. \]
Then $\hocolim_n E \wedge \Gmp{n}$ is a model for the
$\rho$-localisation $E[\rho^{-1}]$ of $E$ (see Lemma \ref{lemm:rho-local-model}). It
follows that its homotopy sheaves are given by
\[ \ul{\pi}_i(E[\rho^{-1}]) = \ul{\pi}_i(E)_*[\rho^{-1}] =: \colim_n \ul{\pi}_i(E)_n. \]
Here the colimit is along multiplication by $\rho$.
(Let us remark here that the homotopy sheaves in $\SH(k)$ are \emph{bigraded},
and so, technically, are those in $\SH(k)[\rho^{-1}]$. However inverting
$\rho$ means that up to canonical isomorphism, the homotopy sheaf is
independent of the second index, so we suppress it.)
It then follows from the descent spectral sequence that in order to prove that
the functor $\SH(k)[\rho^{-1}] \to \SH(k)^\ret[\rho^{-1}]$ is an
equivalence, it is enough to prove that if $F_*$ is a homotopy module (element
in the heart of $\SH(k)$) such that $\rho: F_n \to F_{n+1}$ is an isomorphism
for all $n$ (we call such a homotopy module \emph{$\rho$-stable}),
then $H^n_\ret(X, F_*) = H^n_{Nis}(X, F_*)$ for all $X$ smooth over
$k$. In particular, we need to show that $F_*$ is a sheaf in the real étale
topology. This is actually sufficient, because Nisnevich, Zariski and real étale
cohomology of real étale sheaves all agree \cite[Proposition
19.2.1]{real-and-etale-cohomology}.

This ties in with work of Jacobson and Scheiderer. Recall that
$\ul{\pi}_0(\tunit)_* =
\ul{K}_*^{MW}$, i.e. the zeroth stable motivic homotopy sheaf is unramified
Milnor-Witt $K$-theory. A theorem of Jacobson
\cite{jacobson-fundamental-ideal} together with work of Morel implies that
$\ul{K}_*^{MW}[\rho^{-1}] = \colim_n \ul{I}^n = a_\ret \ul\ZZ$; here $\ul{I}$ is
the sheaf of fundamental ideals. Finally if $F_*$ is a general
$\rho$-stable homotopy module, we use properties of transfers for homotopy
modules together with the structure of $F_*$ as a module over
$\ul{K}_*^{MW}[\rho^{-1}] = a_\ret \ul\ZZ$ to show that $F_*$ is a sheaf in the
real étale topology. This concludes the overview of the proof.

Throughout the article we actually establish all our results for both the stable
motivic homotopy category $\SH(S)$ and the stable $\Aone$-derived category
$D_\Aone(S)$. The proofs in the latter case are essentially always the same as
in the former, so we do not tend to give them. (In fact in some cases proofs
just for
the latter category would be simpler.)

\paragraph{Overview of the article.} In Section \ref{sec:recoll-local-homotopy}
we recall some results from local homotopy theory, including the existence and
basic properties of the homotopy $t$-structure, a general compact generation
criterion and a fully faithfulness result.

In Section \ref{sec:recoll-real-etale} we recall the real étale topology and
establish some supplements.

In Section \ref{sec:recollections-motivic-homotopy} we recall some results about
motivic stable homotopy categories and transfers for finite étale morphisms. In
particular we establish the base change and projection formulas for these.

In Section \ref{sec:recoll-pre-motivic} we recall the formalism of pre-motivic
and motivic categories and how it can be used to establish that a category
satisfies the six functors formalism.

In Section \ref{sec:monoidal-local} we carefully prove some basic facts about
monoidal Bousfield localization.

We judge these five sections as preliminary and the results as not very original.
The ``real work'' is contained in the next three sections. In Section
\ref{sec:jacobson} we review Jacobson's theorem on the colimit of the powers of
the sheaf of fundamental ideals and use it together with our results on
transfers to prove that $\rho$-stable homotopy modules are sheaves in the real
étale topology.

Section \ref{sec:preliminary} contains various preliminary observations and
reductions.

Finally in Section \ref{sec:main-theorems} we carry out the proof as outlined
above.

The remaining three sections contain some applications. In Section
\ref{sec:real-realisation} we show that our functor $\SH(\RR) \to
\SH(\RR)[\rho^{-1}] \wequi \sSH(Spec(\RR)_\ret) \wequi \SH$ coincides with the
real realisation functor. It follows that the $\rho$-inverted stable homotopy
sheaves of $E \in \SH(\RR)$ are just the stable homotopy groups of its real
realisation.

In Section \ref{sec:app1} we collect some consequences for the $\eta$-inverted
sphere. We use that $\tunit[1/2, 1/\rho] \wequi \tunit[1/2, 1/\eta]$. Since the
classical stable stems $\pi_i^s =
\ZZ/2$ for $i = 1, 2$ are $2$-torsion, it follows that
$\ul{\pi}_i(\tunit[1/2,1/\eta])(\RR) = 0$ for $i=1,2$. Since the $\rho$-local
homotopy sheaves are unramified sheaves in the real étale topology, this (more
or less) implies that $\ul{\pi}_i(\tunit[1/2,1/\eta]) = 0$ for $i = 1,2$. This
reproves a result of Röndigs \cite{rondigs2016eta}.

A different but related question is to determine rational motivic stable
homotopy theory. By a recent result of Ananyevskiy-Levine-Panin
\cite{levine2015witt} we have $\SH(k)_\QQ^- \wequi \DM_W(k,\QQ)$, where the
right hand side denotes a category of rational Witt-motives. Our results show
easily that $\DM_W(k, \ZZ[1/2]) \wequi D_\Aone(k, \ZZ[1/2])^- \wequi
D(Spec(k)_\ret, \ZZ[1/2])$ and more generally that $D_\Aone(k, \ZZ)[1/\rho] \wequi
D(Spec(k)_\ret)$. By the same proof as in classical rational stable homotopy
theory we have $\SH(k)_\QQ^- \wequi D_\Aone(k, \QQ)^-$, and so we consider our
results as one version of an integral strengthening of the result of
Ananyevskiy-Levine-Panin.

In Section \ref{sec:app2} we collect some applications to the rigidity problem.
A sheaf $F$ on $Sm(k)$ is called \emph{rigid} if for every essentially smooth,
Henselian local scheme $X$ with closed point $x$ we have $F(X) = F(x)$. For
example, sheaves with transfers in the sense of Voevodsky which are of torsion
prime to the characteristic of the perfect base field are rigid (see
\cite[Theorem 4.4]{suslin1996singular}). Our
results imply that the homotopy sheaves of any $E \in \SH(k)[\rho^{-1}]$ are
real étale sheaves extended from the small real étale site of $k$. One might
already call this a rigidity result, but it is also not hard to see (and we
show) that all such sheaves are rigid in the above sense. As an application, we
show that the motivic stable homotopy sheaves $\ul{\pi}_i(\tunit)_0[1/e]$ are all
rigid, where $e$ is the exponential characteristic.
This ties up a loose end of the author's PhD thesis.

\paragraph{Acknowledgements.} This paper owes a huge debt to a number of people.
In May 2016 Denis-Charles Cisinski taught a mini-course on motives in Essen. As
a result the author realised that he could extend his theorem from perfect base
fields to fairly general base schemes.
In particular he learned the pattern of the proof that $D(S_{et}, \ZZ/p)$ is
equivalent to $\DM_h(S, \ZZ/p)$. In a lot of ways our proof is a variant of that
one. A similar strategy is followed in the Cisinski-Déglise article
\cite{cisinski2013etale} on which we also rely heavily both in spirit and in
practice.

Just like in that article, many of our results are relatively
straightforward consequences of difficult theorems in real étale cohomology
established by Scheiderer and other semi-algebraic topologists.

The importance of the Voevodsky/Ayoub/Cisinski-Déglise approach to the six
functors formalism for our article also cannot be overstated.

Discussions with Fabien Morel, Oliver Röndigs, Marc Hoyois, Marco
Schlichting, and Markus Spitzweck about early versions of this work were also
influential to its current form.

The author would further like to thank Denis-Charles Cisinski for carefully
reading a draft of this work and pointing out several mistakes, and Elden
Elmanto, Daniel Harrer, Marc Levine and Maria Yakerson for providing comments.

\paragraph{Notation.} If $S$ is a scheme, we denote the motivic stable homotopy
category by $\SH(S)$. We denote the $S^1$-stable motivic homotopy category (i.e.
where $\Gm$ has not been inverted yet) by $\SH^{S^1}(S)$. If $\mathcal{X}$ is a topos
or site, we denote by $\sSH(\mathcal{X})$ the associated stable homotopy
category, see Section \ref{sec:recoll-local-homotopy}. In particular
$\sSH(S_\ret), \sSH(Sm(S)_\ret)$ and $\SH(S)^\ret$ should be carefully
distinguished: the first is the stable homotopy category of the small rét-site
on $S$, the second is the stable homotopy category of the site of all smooth
schemes, with the rét-topology, and the latter is the rét-localization of the
motivic stable homotopy category. This last category is $\Aone$-local and
$\Gm$-stable, whereas the second category is neither, and these notions do not
even make sense for the first category.

The classical stable
homotopy category will still be denoted by $\SH$.

We denote the unit of a monoidal category $\mathcal{C}$ by $\tunit_\mathcal{C}$
or just by $\tunit$, if $\mathcal{C}$ is clear from the context. Thus if
$\mathcal{C}$ is a stable homotopy category of some sort, then $\tunit$ is the
sphere spectrum.

\section{Recollections on Local Homotopy Theory}
\label{sec:recoll-local-homotopy}

If $(\mathcal{C}, \tau)$ is a Grothendieck site, we can consider the associated category
$Shv(\mathcal{C}_\tau)$ of sheaves (a topos), the category $sPre(\mathcal{C})$
of \emph{simplicial presheaves} on $\mathcal{C}$, as well as the categories
$\mathcal{SH}(\mathcal{C})$ of presheaves of spectra and $C(\mathcal{C})$ of
presheaves of complexes of abelian groups on $\mathcal{C}$. The latter three
categories carry various local model structures, in particular the injective and
the projective one \cite{jardine-local-homotopy}. We denote the homotopy category of
$\mathcal{SH}(\mathcal{C}_\tau)$ by $\sSH(\mathcal{C}_\tau)$ and the homotopy
category of $C(\mathcal{C}_\tau)$ by $D(\mathcal{C}_\tau)$.

It is also possible to model $\sSH(\mathcal{C}_\tau)$ and so on by
\emph{sheaves}. For this, let $sShv(\mathcal{C}_\tau)$ denote the category of
sheaves of simplicial sets, and similarly $\mathcal{SH}^s(\mathcal{C}_\tau)$ the
category of sheaves of spectra, and $C^s(\mathcal{C}_\tau)$ the category of
sheaves of chain complexes. (Here we mean sheaves in the 1-categorical sense, so
this category is equivalent to the category of chain complexes of sheaves of
abelian groups, and similarly for the spectra.)
These also afford local model structures, and
$Ho(sShv(\mathcal{C}_\tau)) \wequi Ho(sPre(\mathcal{C}_\tau))$, and so on.

Given a functor $f^*: \mathcal{C} \to \mathcal{D}$, there is an induced
restriction functor $f_*: Pre(\mathcal{D}) \to Pre(\mathcal{C})$, where
$Pre(\mathcal{C})$ denotes the category of presheaves (of sets) on $\mathcal{C}$
(and similarly for $\mathcal{D}$). The functor
$f_*$ has a left adjoint $f^*: Pre(\mathcal{C}) \to Pre(\mathcal{D})$. It is in
fact the left Kan extension of $f^*: \mathcal{C} \to \mathcal{D}$.

If $\mathcal{C}, \mathcal{D}$ are sites the functor $f^*$ is called continuous
if $f_*: Pre(\mathcal{D}) \to Pre(\mathcal{C})$ preserves sheaves. In this case
the induced functor $f_*: Shv(\mathcal{D}) \to Shv(\mathcal{C})$ has a left
adjoint still denoted $f^*: Shv(\mathcal{C}) \to Shv(\mathcal{D})$. If this
induced functor is left exact (commutes with finite limits) then $f$ is called a
\emph{geometric morphism}.

More generally, an adjunction $f^*: Shv(\mathcal{C}) \leftrightarrows
Shv(\mathcal{D}): f_*$ (where $f^* \vdash f_*$ does not necessarily come from a
functor $f^*: \mathcal{C} \to \mathcal{D}$)
is called a geometric morphism if $f^*$ preserves finite
limits.

If $f: \mathcal{C} \to \mathcal{D}$ is any functor, then there are induced
adjunctions $f^*: sPre(\mathcal{C}) \leftrightarrows sPre(\mathcal{D}): f_*$,
and similarly for spectra and chain complexes. Similarly if $f^*: Shv(\mathcal{C})
\leftrightarrows Shv(\mathcal{D}): f_*$ is any adjunction, then there are
induced adjunctions $f^*: sShv(\mathcal{C}) \leftrightarrows sShv(\mathcal{D}):
f_*$, and so on. If $f^* \vdash f_*$ is a geometric morphism in either of the
above senses, then the induced adjunctions on presheaves (sheaves) of simplicial
sets, spectra, and chain complexes are Quillen adjunctions in the local model
structure \cite[Section 5.3]{jardine-local-homotopy} \cite[Theorem
1.18]{cisinski2009local}.

The above discussion allows us to prove the following useful result.

\begin{lemm} \label{lemm:loc-hmtpy-fully-faithful}
Let $f^*: Shv(\mathcal{C}) \leftrightarrows Shv(\mathcal{D}): f_*$ be a
geometric morphism such that $f^*$ is fully faithful and $f_*$ preserves
colimits.

Then the induced functors
\[ Lf^*: \sSH(\mathcal{C}) \to \sSH(\mathcal{D}) \]
and
\[ Lf^*: D(\mathcal{C}) \to D(\mathcal{D}) \]
are fully faithful.
\end{lemm}
The same result also holds for $Lf^*: Ho(sPre(\mathcal{C})) \to
Ho(sPre(\mathcal{D}))$, with the same proof.
\begin{proof}
We give the proof for the derived categories, it is the same for spectra.

Since $f_*$ preserves colimits it affords a right adjoint $f^!$.
Then $f_* \vdash f^!$ is a geometric morphism in the opposite direction (note
that $f_*$ preserves finite limits, and in fact all limits, since it is a right
adjoint) and consequently $f_*$ is \emph{bi}-Quillen. It follows that $f_*:
C^s(\mathcal{D}) \to C^s(\mathcal{C})$ preserves weak equivalences, and
consequently coincides (up to weak equivalence) with its derived functor.

Now to show that $Lf^*$ is fully faithful we need to show that $Rf_* Lf^* \wequi
\id$. But $Rf_* \wequi f_*$ since $f_*$ is bi-Quillen.
Let $E \in C^s(\mathcal{C})$ be cofibrant. Then $Lf^* E
\wequi f^*E$ and consequently $Rf_* Lf^* E \wequi f_* f^* E$. Since $f^*$ is fully
faithful we have $f_* f^* E \iso E$. This concludes the proof.
\end{proof}

We will also make use of $t$-structures. We shall use \emph{homological}
notation for $t$-structures \cite[Definition 1.2.1.1]{lurie-ha}. Briefly, a
$t$-structure on a triangulated category $\mathcal{C}$ consists of two (strictly
full) subcategories $\mathcal{C}_{\ge 0}$ and $\mathcal{C}_{\le 0}$, satisfying
various axioms. We put $\mathcal{C}_{\ge n} = \mathcal{C}_{\ge 0}[n]$ and
$\mathcal{C}_{\le n} = \mathcal{C}_{\le 0}[n]$. One then has $\mathcal{C}_{\ge
n+1} \subset \mathcal{C}_{\ge n}$ and $\mathcal{C}_{\le n} \subset
\mathcal{C}_{\le n+1}$ and $[\mathcal{C}_{\ge n+1}, \mathcal{C}_{\le n}] = 0$.
In fact $E \in \mathcal{C}_{\ge n+1}$ if and only if for all $F \in
\mathcal{C}_{\le n}$ we have $[E, F] = 0$, and vice versa.
The inclusion $\mathcal{C}_{\ge n} \hookrightarrow \mathcal{C}$ has a right
adjoint which we denote $E \mapsto E_{\ge n}$, and the inclusion
$\mathcal{C}_{\le n} \hookrightarrow \mathcal{C}$ has a left adjoint which we
denote $E \mapsto E_{\le n}$. The adjunctions furnish map $E_{\ge n+1} \to E \to
E_{\le n}$ and this extends to a distinguished triangle in a unique and
functorial way. The intersection $\mathcal{C}^\heart := \mathcal{C}_{\ge 0} \cap
\mathcal{C}_{\le 0}$ called the \emph{heart}. It is an abelian category. We put
$\pi^\mathcal{C}_0(E) = (E_{\le 0})_{\ge 0} \wequi (E_{\ge 0})_{\le 0} \in
\mathcal{C}^\heart$ and $\pi^\mathcal{C}_i(E) = \pi_0^\mathcal{C}(E[i])$. Then
$\pi_*^\mathcal{C}$ is a homological functor on $\mathcal{C}$. The $t$-structure
is called \emph{non-degenerate} if $\pi_i^\mathcal{C}(E) = 0$ implies that $E
\wequi 0$.

By a $t$-category we mean a triangulated category with a fixed $t$-structure.

Suppose that $(\mathcal{C}, \tau)$ is a site.  Let for $E \in
\mathcal{SH}(\mathcal{C}_\tau)$ and $i \in \ZZ$
the sheaf $\ul{\pi}_i(E) \in Shv(\mathcal{C}_\tau)$ be
defined as the sheaf associated with the presheaf $\mathcal{C} \ni X \mapsto
\pi_i(E(X))$. Here we view $E$ as a presheaf of spectra. By definition, local
weak equivalences of spectra induce isomorphisms on $\ul{\pi}_i$, so
$\ul{\pi}_i(E)$ is well-defined for $E \in \sSH(\mathcal{C}_\tau)$. This is a
sheaf of \emph{abelian groups}. Put
\begin{gather*}
 \sSH(\mathcal{C}_\tau)_{\ge 0} = \{E \in \sSH(\mathcal{C}_\tau): \ul{\pi}_i(E) = 0 \text{ for $i < 0$} \} \\
 \sSH(\mathcal{C}_\tau)_{\le 0} = \{E \in \sSH(\mathcal{C}_\tau): \ul{\pi}_i(E) = 0 \text{ for $i > 0$} \}.
\end{gather*}

We define similarly for $E \in D(\mathcal{C}_\tau)$ the sheaf $\ul{h}_i(E)$, and
then the subcategories $D(\mathcal{C}_\tau)_{\ge 0}, D(\mathcal{C}_\tau)_{\le
0}.$

\begin{lemm} \label{lemm:SH-t-structure}
If $(\mathcal{C}, \tau)$ is a Grothendieck site, then the above construction
provides $\sSH(\mathcal{C}_\tau)$ with a non-degenerate $t$-structure. The
functor $\ul{\pi}_0: \sSH(\mathcal{C}_\tau)^\heart \to Shv(\mathcal{C}_\tau)$ is
an equivalence of categories. Moreover let $F \in Shv(\mathcal{C}_\tau) \wequi
\sSH(\mathcal{C})^\heart$. Then for $X \in \mathcal{C}$ there is a natural
isomorphism $[\Sigma^\infty X_+, F[n]] = H^n_\tau(X, F)$.

Similar statements hold for $D(\mathcal{C}_\tau)$ in place of
$\sSH(\mathcal{C}_\tau)$.
\end{lemm}
\begin{proof}
For derived categories, this result is classical. For $\sSH(\mathcal{C}_\tau)$,
the result is also fairly well known, but the author does not know an explicit
reference, so we sketch a proof.

Note that there is a Quillen adjunction (in the
local model structures)
\[ \Sigma^\infty: sPre(\mathcal{C}_\tau)_* \leftrightarrows \mathcal{SH}(\mathcal{C}_\tau): \Omega^\infty. \]
By direct computation using the above adjunction, we find that $\ul{\pi}_i(\Omega^\infty E) =
\ul{\pi}_i(E)$, for $E \in \sSH(\mathcal{C}_\tau)$ and $i \ge 0$.

By \cite[Proposition 1.4.3.4
and Remark 1.4.3.5]{lurie-ha} the category $\sSH(\mathcal{C}_\tau)$ admits a
$t$-structure\footnote{The author would like to thank Saul Glasman for pointing
out this reference.}, where $E \in \sSH(\mathcal{C}_\tau)_{\le 0}$ if and only if
$\Omega^\infty(E) \simeq *$, and the subcategory $\sSH(\mathcal{C}_\tau)_{\ge
0}$ is generated under homotopy colimits and extensions by $\Sigma^\infty
\mathcal{C}_+$. We first need to show that this is the $t$-structure we want,
i.e. that the positive and negative parts are determined by vanishing of
homotopy sheaves. Since $\ul{\pi}_i(\Omega^\infty E) = \ul{\pi}_i(E)$, this is
correct for the negative part. I claim that if $E \in \sSH(\mathcal{C}_\tau)_{\ge
0}$, then $\ul{\pi}_i(E) = 0$ for $i < 0$. If $X \in sPre(\mathcal{C}_\tau)_*$,
then $\ul{\pi}_i(\Sigma^\infty X) = 0$ for $i<0$ by direct computation. It thus
remains to show that the subcategory of $E \in \sSH(\mathcal{C}_\tau)$ with
$\ul{\pi}_i(E) = 0$ for $i < 0$ is closed under homotopy colimits and
extensions. For extensions this is clear. Homotopy colimits are generated by
pushouts and filtered colimits \cite[Propositions 4.4.2.6 and
4.4.2.7]{lurie-htt}, so we need only deal with cones and filtered colimits. For
cones this is again clear, and for filtered colmits it holds because homotopy
groups of spectra commute with filtered colimits, and hence the same is true for
homotopy sheaves (see the proof of Corollary \ref{corr:SH-compact-ob} for more
details on this). This proves the claim. Conversely, let $E \in
\sSH(\mathcal{C}_\tau)$ with $\ul{\pi}_i(E) = 0$ for $i < 0$. Consider the
decomposition $E_{\ge 0} \to E \to E_{<0}$. Then $\ul{\pi}_i(E_{\ge 0}) = 0$ for
$i < 0$, so $0 = \ul{\pi}_i(E) = \ul{\pi}_i(E_{<0})$ for $i < 0$. It follows
that $E_{<0} \simeq 0$ and so $E \simeq E_{\ge 0} \in \sSH(E)_{\ge 0}$.

The $t$-structure is non-degenerate because it is defined in terms of homotopy
sheaves, and homotopy sheaves detect weak equivalences by definition.

We have an adjunction
\[ M: \sSH(\mathcal{C}_\tau) \leftrightarrows D(\mathcal{C}_\tau): U. \]
By construction $U$ is $t$-exact and thus $M$ is right $t$-exact. Consider the
induced adjunction
\[ M^\heart: \sSH(\mathcal{C}_\tau)^\heart \leftrightarrows D(\mathcal{C}_\tau)^\heart: U. \]
By direct computation using the classical Hurewicz isomorphism (and the above
adjunction), $\ul{\pi}_0(UME) = \ul{\pi}_0(E)$ if $E \in
\sSH(\mathcal{C}_\tau)_{\ge 0}$. It follows that $UM^\heart \simeq \id$. Since $U$ is
faithful by definition, from this we deduce that $M^\heart U \simeq \id$ as
well. Thus $\sSH(\mathcal{C}_\tau)^\heart \simeq D(\mathcal{C}_\tau)^\heart
\simeq Shv(\mathcal{C}_\tau)$, the latter equivalence being classical. Finally
if $X \in \mathcal{C}$ and $F \in Shv(\mathcal{C}_\tau)$ then $[\Sigma^\infty
X_+, F[n]] = [\Sigma^\infty X_+, UF[n]] = H^n_\tau(X,F)$, the first equality by
definition and the second by adjunction and the same result in
$D(\mathcal{C}_\tau)$.
\end{proof}

\begin{corr} \label{corr:SH-compact-ob}
Let $(\mathcal{C}, \tau)$ be a Grothendieck site.

\begin{enumerate}[(1)]
\item Let $X \in \mathcal{C}$. If
  $\tau$-cohomology on $X$ commutes with filtered colimits of sheaves and the
  $\tau$-cohomological dimension of $X$ is finite, then $\Sigma^\infty X_+ \in
  \sSH(\mathcal{C}_\tau)$ is a compact object.

\item For any collection $E_i \in \sSH(\mathcal{C})$ and $j \in \ZZ$ we have
  $\ul{\pi}_j(\bigoplus_i E_i) = \bigoplus_i \ul{\pi}_j(E_i)$.
\end{enumerate}

Similarly for $D(\mathcal{C}_\tau)$.
\end{corr}
\begin{proof}
Let us show that (1) reduces to (2).
For $E \in \sSH(\mathcal{C}_\tau)$ there is a conditionally convergent spectral sequence
\[ H^p_\tau(X, \ul{\pi}_{-q} E) \Rightarrow [X, E[p+q]]. \]
Under our assumptions on the cohomological dimension of $X$, it converges
strongly to the right hand side. Under the assumption of commutation of
cohomology with filtered colimits, by spectral sequence comparison, it thus
suffices to show that for $E_i \in \sSH(\mathcal{C}_\tau)$ we have $\ul{\pi}_n(\bigoplus_i E_i)
= \bigoplus_i \ul{\pi}_n E_i$.

Now we prove (2).
For $E \in \mathcal{SH}(\mathcal{C})$ write $\ul{\pi}_j^p(E)(X) = \pi_j(E(X))$; this
defines a presheaf of abelian groups on $\mathcal{C}$. By definition $\ul{\pi}_j(E) = a_\tau
\ul{\pi}_j^p(E)$. Let $\{E_i\}_i \in \mathcal{SH}(\mathcal{C})$. Then
$\ul{\pi}_j^p(\bigoplus_i E_i) = \bigoplus_i \ul{\pi}_j^p(E_i)$, since homotopy
groups of spectra commute with filtered colimits. We may assume that
all the $E_i$ are cofibrant, so their presheaf direct sum coincides with the
derived direct sum. In this case it remains to show that
\[ a_\tau \bigoplus_i \ul{\pi}_j^p(E_i) \iso \bigoplus_i a_\tau \ul{\pi}_j^p(E_i). \]
(Note that here we write $\bigoplus_i$ for both direct sums of presheaves and
direct sums of sheaves, depending on whether the terms on the right are
presheaves or sheaves.) But this holds for any collection of presheaves on any
site (both sides satisfy the same universal property).

The proof for $D$ is the same.
\end{proof}

We can enhance the functoriality of the $\sSH$ construction as follows. Recall
that a triangulated functor $F: \mathcal{C} \to \mathcal{D}$ between
$t$-categories is called right (respectively left) $t$-exact if
$F(\mathcal{C}_{\ge 0}) \subset \mathcal{D}_{\ge 0}$ (respectively
$F(\mathcal{C}_{\le 0}) \subset \mathcal{D}_{\le 0}$). The functor is called
$t$-exact if it is both left and right $t$-exact.

\begin{lemm} \label{lemm:SH-functor-t}
Let $f^*: Shv(\mathcal{C}) \leftrightarrows Shv(\mathcal{D}): f_*$ be a
geometric morphism, where $Shv(\mathcal{D})$ has enough points.
Then in the adjunction
\[ Lf^*: \sSH(\mathcal{C}) \leftrightarrows \sSH(\mathcal{D}): Rf_* \]
the left adjoint $Lf^*$ is $t$-exact, the right adjoint $Rf_*$ is
left $t$-exact, and the induced functors
\[ (Lf^*)^\heart: \sSH(\mathcal{C})^\heart \leftrightarrows \sSH(\mathcal{D})^\heart: (Rf_*)^\heart \]
coincide (under the identification from Lemma \ref{lemm:SH-t-structure}) with
$f^* \vdash f_*$.

Similar statements hold for $D$ in place of $\sSH$.
\end{lemm}
The author contends that the assumption that $\mathcal{D}$ has enough points is
not really necessary. See also \cite[Remark 6.5.1.4]{lurie-htt}.
\begin{proof}
Certainly $Rf_*$ is left $t$-exact if $Lf^*$ is $t$-exact by adjunction, and
$(Rf_*)^\heart$ is right adjoint to $(Lf^*)^\heart$, so it suffices to prove the
claims for $Lf^*$.

Since $\mathcal{D}$ has enough points,
it is then enough to assume that $Shv(\mathcal{D}) = Set$. (Indeed let $p: Set
\to Shv(\mathcal{D})$ be a point; we will have
\[ p^* \ul{\pi}_i(Lf^* E) = \pi_i(Lp^*Lf^*E) = p^*f^*\ul{\pi}_i E \]
for all $E \in \sSH(\mathcal{C})$ by applying the reduced case to $p$ and $fp$
which are points of $\mathcal{D}$ and $\mathcal{C}$, respectively. Since
$\mathcal{D}$ has enough points it follows that $\ul{\pi}_i(Lf^*E) =
f^*\ul{\pi}_i(E)$, as was to be shown.)

Let $p^*: Shv(\mathcal{C}) \leftrightarrows
Set: p_*$ be a point of $\mathcal{C}$. Then $p^*$ corresponds to a pro-object in
$\mathcal{C}$, which is to say that there is a filtered family $X_\alpha \in
\mathcal{C}$ such that for $F \in Shv(\mathcal{C})$ we have
$p^*(F) = \colim_\alpha F(X_\alpha)$
\cite[Proposition 1.4 and Remark 1.5]{gabber2015points}.

It follows that for $E \in \mathcal{SH}^s(\mathcal{C})$ we have
\[ \pi_i(p^*E) = \pi_i(\colim_\alpha E(X_\alpha)) \iso \colim_\alpha
       \pi_i(E(X_\alpha)) = p^* \ul{\pi}_i(E), \]
where the isomorphism in the middle holds because homotopy groups commute with
filtered colimits of spectra. In particular $p^*$ preserves weak
equivalences and so $p^* \wequi Lp^*$. Thus the previous equation is precisely what we
intended to prove.
\end{proof}

\section{Recollections on Real Étale Cohomology}
\label{sec:recoll-real-etale}

If $X$ is a scheme, let $R(X)$ be the set of pairs $(x, p)$ where $x \in X$ and
$p$ is an ordering of the residue field $k(x)$. For a ring $A$ we put $Sper(A) =
R(Spec(A))$. A family of morphisms $\{\alpha_i: X_i \to X\}_{i \in I}$ is called
a \emph{real étale covering} if each $\alpha$ is étale and $R(X) = \cup_i
\alpha(R(X_i))$. (Note that for $(x, p) \in X_i$ the extension
$k(x)/k(\alpha(x))$ defines by restriction an ordering of $k(\alpha(x))$.) The real
étale coverings define a topology on all schemes \cite[(1.1)]{real-and-etale-cohomology}
called the \emph{real étale topology}. We often
abbreviate this name to ``rét-topology''.

For a scheme $X$, we let $X_{\ret}$ denote the small real étale site on $X$ and
$Sm(X)_{\ret}$ the site of smooth (separated, finite type) schemes over $X$ with
the real étale topology. If $f: X \to Y$ is any morphism of schemes, we get the
usual base change functors $f^*: Y_{\ret} \to X_{\ret}$ and $f^*: Sm(Y) \to
Sm(X)$. Also the natural inclusion $e: X_\ret \to Sm(X)$ induces an adjunction
$e^p: Pre(X_\ret) \leftrightarrows Pre(Sm(X)): r = e_*$.

\begin{lemm} \label{lemm:ret-extension-fully-faithful}
If $X$ is a scheme, the above adjunction induces a geometric morphism
$e: Shv(X_\ret) \leftrightarrows Shv(Sm(X)_\ret) : r$ where $e$ is fully
faithful and $r$ preserves colimits.
\end{lemm}
\begin{proof}
The functor $r$ is restriction and $e$ is left Kan extension. Since $e$
preserves covers, $r$ preserves sheaves. Moreover $r$ commutes with taking the
associated sheaf, because every cover of $Y \in X_{\ret}$ in $Sm(X)$ comes from a
cover in $X_{\ret}$ (because étale morphisms are stable under composition).
It follows that $r$ commutes with colimits. Since $e:
X_{\ret} \to Sm(X)_{\ret}$ preserves pullbacks (and $X_\ret$ has pullbacks!), the
adjunction is a geometric morphism \cite[Tag 00X6]{stacks-project}.
In order to see that $e$ is fully faithful,
i.e. $F \to reF$ an isomorphism for every $F \in Shv(X_\ret)$,
we note that for the presheaf adjunction $e^p: Pre(X_{\ret}) \leftrightarrows
Pre(Sm(k)) : r$ we have $re^pF = F$. Indeed this holds for $F$ representable by
definition, every sheaf is a colimit of representables, and $e^p$ and $f$ both
commute with taking colimits. Finally note that for a sheaf $F$ we have $e F =
a_{\ret} e^p F$ and thus $re F = r a_{\ret} e^p F = a_{\ret} re^p F = a_{\ret} F =
F$, where we have used again that $r$ commutes with taking the associated sheaf.
\end{proof}

\begin{corr} \label{corr:Le-exact-ff}
If $X$ is a scheme, the induced derived functor $Le: \sSH(X_\ret) \to
\sSH(Sm(X)_\ret)$ is $t$-exact and fully faithful. Similarly for $D$ in place of
$\sSH$.
\end{corr}
\begin{proof}
The functor is fully faithful by Lemmas \ref{lemm:ret-extension-fully-faithful}
and \ref{lemm:loc-hmtpy-fully-faithful}. It is $t$-exact by Lemma
\ref{lemm:SH-functor-t}.
\end{proof}

\begin{lemm} \label{lemm:f*-SH-Xret-exact}
If $f: X \to Y$ is a morphism of schemes, then the induced functor $f^*: Y_\ret
\to X_\ret$ is the left adjoint of a geometric morphism of sites. Moreover the
derived functor
\[ Lf^*: \sSH(Y_\ret) \to \sSH(X_\ret) \]
is $t$-exact, and similarly for $Lf^*: D(Y_\ret) \to D(X_\ret)$.
\end{lemm}
\begin{proof}
The ``moreover'' part follows from Lemma \ref{lemm:SH-functor-t}.

Since $f^*: Y_\ret \to X_\ret$ preserves covers $f_*: Pre(X_\ret) \to
Pre(Y_\ret)$ preserves sheaves and the morphism is continuous. It is a geometric morphism
of sites because $f^*$ preserves pullbacks \cite[Tag 00X6]{stacks-project}.
\end{proof}

If $X$ is a scheme, there is the natural map $X \to X \times \Aone$
corresponding to the point $0 \in \Aone$. Similarly there is the natural map $X
\coprod X \to X \times (\Aone \setminus 0)$ corresponding to the points $\pm 1
\in \Aone \setminus 0$.

\begin{thm} \label{thm:ret-htpy-rho-inv}
Let $X$ be a scheme and $F \in Shv(X_\ret)$. Then for any $p \ge 0$ the natural maps $X \to X \times
\Aone$ and $X \coprod X \to X \times (\Aone \setminus 0)$ induce isomorphisms
\[ H^p_\ret(X \times \Aone, F) \to H^p_\ret(X, F) \]
\[ H^p_\ret(X \times (\Aone \setminus 0), F) \to H^p_\ret(X, F) \oplus H^p_\ret(X, F). \]
\end{thm}
\begin{proof}
The first statement is homotopy invariance, see \cite[Example
16.7.2]{real-and-etale-cohomology}.

For the second statement, we follow closely that proof. Let $f\colon X \coprod X \to
X \times (\Aone \setminus 0)$ be the canonical map. It suffices to show that
$R^n f_* F = 0$ for $n > 0$ and $R^0 f_* F = F$, where we identify $F$ with its
pullback to $X \coprod X$ and $X \times (\Aone \setminus 0)$ for notational
convenience. All of these statements are local on $X$, so we may assume that $X$
is affine.

Then one may assume that $F$ is constructible (since $\ret$-cohomology commutes
with filtered colimits of sheaves, and all sheaves on a spectral space are
filtered colimits of constructible sheaves; see again \emph{loc. cit.}). Next,
writing $X = Spec(A)$ as the inverse limit of the filtering system $Spec(A')$,
with $A' \subset A$ finitely generated over $\ZZ$, and using Proposition (A.9) of
\emph{loc.  cit.}, we may assume that $X$ is of finite type over $\ZZ$.

But $Sper(\ZZ) = Sper(\QQ) = Sper(\RR)$, whence $H^p_\ret(X, F) = H^p_\ret(X
\times_\ZZ \RR, F)$, so we may assume that $X$ is of finite type over $\RR$.

We may further assume that $F = M_Z$ is the constant sheaf on a closed,
constructible subset of $X$ (Proposition (A.6) of \emph{loc. cit.}).

It is thus enough to prove the analog of our result for an affine semi-algebraic
space $X$ over $\RR$ and $F = M$ a constant sheaf. But then $H^*_\ret(X, M) =
H^*_{sing}(X(\RR), M)$ \cite[Theorem II.5.7]{delfs1991homology} and so on, so this is obvious.
\end{proof}

\begin{thm}[(Proper Base Change)] \label{thm:ret-proper-basechange}
Consider a cartesian square of schemes
\begin{equation*}
\begin{CD}
X'     @>g'>> X \\
@Vf'VV       @VVfV \\
Y'     @>g>> Y,
\end{CD}
\end{equation*}
with $f$ proper and $Y$ finite-dimensional Noetherian.
Then for any $E \in \sSH(X_\ret)$ (respectively $E \in
D(X_\ret)$) the canonical map
\[ g^* Rf_* (E) \to Rf'_* g'^* (E) \]
is a weak equivalence.
\end{thm}
\begin{proof}
We prove the claim for $\sSH$, the proof we give will work just as well for $D$.
We proceed in several steps.

\paragraph{Step 0.} If $g$ is étale, then the claim follows from the observation
that $f^* g_\# = g'_\# f'^*$.

\paragraph{Step 1.} If $f\colon X \to Y$ is any morphism and $E \in \sSH(X_\ret)$,
then there is a conditionally convergent spectral sequence
\[ E_2^{pq} = R^p f_* \ul{\pi}_{-q} E \Rightarrow \ul{\pi}_{-p-q}(Rf_* E). \]

For this, let $E \in Spt(X_\ret)$ also denote a fibrant model. Then $Rf_* E
\simeq f_* E$ and for $U \in Y_\ret$ we have $f_*(E)(U) = E(f^*U)$. Since $E$ is
fibrant there is a conditionally convergent descent spectral sequence
\[ H^p(f^*U, \ul{\pi}_{-q}(E)) \Rightarrow \pi_{-p-q}(E(f^*U)). \]
By varying $U$, this yields a presheaf of spectral sequences on $Y_\ret$.
Equivalently, this is a spectral sequence of presheaves. Taking the
associated sheaf on both sides we obtain a conditionally convergent spectral
sequence
\[ a_\ret H^p_\ret(f^*\bullet, \ul{\pi}_{-q}(E)) \Rightarrow \ul{\pi}_{-p-q}(f_* E). \]
It remains to see that $a_\ret H^p_\ret(f^*\bullet, F) = R^p f_*
F$, for any sheaf $F$ on $X_\ret$. For this we view $F \in
D(X_\ret)^\heart$. Then by definition $R^p f_* F = \ul{\pi}_{-p} Rf_* F$.
Repeating the above argument with $D(X_\ret)$ in place of $\sSH(X_\ret)$,
we obtain a conditionally convergent spectral
sequence
\[ a_\ret H^p_\ret(f^*\bullet, \ul{\pi}_{-q}F) \Rightarrow R^{p+q}f_* F. \]
Since $\ul{\pi}_{-q}F = 0$ for $q \ne 0$ this spectral sequence converges
strongly, yielding the desired identification.

\paragraph{Step 2.} If $f$ is proper and of
relative dimension at most $n$, then for $F \in Shv(X_\ret)$ and
$p>n$ we have $R^p f_* F = 0$.

Indeed in this situation, by the proper base change
theorem in real étale cohomology \cite[Theorem
16.2]{real-and-etale-cohomology}, for any real closed point
$y \to Y$ we get $(R^pf_* F)_y = H^p_\ret(X_y, F|_{X_y})$. Since real closed
fields are the stalks of the rét-topology, in order for a sheaf $G \in
Shv(Y_\ret)$ to be zero it is necessary and sufficient that $G_y = 0$ for all
such $y$.
But real étale cohomological dimension is bounded by
Krull dimension \cite[Theorem 7.6]{real-and-etale-cohomology}, so we find that
$R^pf_* F = 0$ for $p > n$, as claimed.

\paragraph{Conclusion of proof.}
Since isomorphism in $\sSH(Y'_\ret)$ is local on $Y'$, it is an easy consequence
of step 0 that we may assume that $Y'$ is quasi-compact (e.g. affine). Then $f'$ is of
bounded relative dimension (being of finite type).

Now let $E \in \sSH(X_\ret)$.
By $t$-exactness of $g^*$ and $g'^*$ we get from step 1 conditionally convergent spectral
sequences
\[ g^* R^p f_* \ul{\pi}_{-q} E \Rightarrow \ul{\pi}_{-p-q}(g^* Rf_* E) \]
and
\[ R^p f'_* g'^* \ul{\pi}_{-q} E \Rightarrow \ul{\pi}_{-p-q}(Rf'_* g'^* E). \]
The exchange transformation $g^* Rf_* (E) \to Rf'_* g'^* (E)$ induces a morphism
of spectral sequences (i.e. respecting the differentials and filtrations).
By proper base change for sheaves, we have $g^* R^p f_* \iso R^p f'_* g'^*$.
Thus the two spectral sequences are isomorphic. By step 2 the second one
converges strongly, and hence so does the first. Thus the result follows from
spectral sequence comparison.
\end{proof}

\paragraph{Remark.} The only place in the above proof where we have used the
assumption on $Y$ is in step 1, namely in the construction of the conditionally
convergent spectral sequence \[ R^p f_* \ul{\pi}_{-q} E \Rightarrow
\ul{\pi}_{-p-q}(Rf_* E). \] The author does not know how to construct such a
spectral sequence in general. He nonetheless contends that the proper base
change theorem should be true without assumptions on $Y$, but perhaps a
different proof is needed.

\paragraph{Remark.} In the above proof we deduce proper base change for spectra
and unbounded complexes from proper base change for bounded complexes. Since we
are dealing with hypercomplete toposes, this is not tautological; see for example
\cite[Counterexample 6.5.4.2 and Remark 6.5.4.3]{lurie-htt}. The crucial
property which seems to make the proof work is encapsulated in step 2 and might
be phrased as ``a proper morphism is locally of finite relative
rét-cohomological dimension''. The same is true in étale (instead of real étale)
cohomology and this seems to be what the proof of proper base change for unbounded
étale complexes \cite[Theorem 1.2.1]{cisinski2013etale} ultimately rests on, in
the guise of \cite[Lemma 1.1.7]{cisinski2013etale}. This fails for a general
proper morphism of topological spaces (consider for example an infinite product of compact
positive dimensional spaces mapping to the point).

\section{Recollections on Motivic Homotopy Theory}
\label{sec:recollections-motivic-homotopy}

We denote the stable motivic homotopy category over a base scheme $X$
\cite{ayoub2007six} by
$\SH(X)$, and the stable $\Aone$-derived category over $X$ \cite[Section
5.3]{triangulated-mixed-motives} by $D_\Aone(X)$.  
We write $\tunit_X \in \SH(X)$ for the monoidal unit. If the context is clear we
may just write $\tunit$.

Let $f: Y \to X$ be a finite étale morphism of schemes. Then in the category
$\SH(X)$ we have an induced morphism $f: f_\# \tunit_Y \to \tunit_X$ and
consequently $D(f): D(\tunit_X) \to D(f_\# \tunit_Y)$. Here $DE := \iHom(E,
\tunit)$. Now in fact whenever $f:
Y \to X$ is smooth proper then $D(f_\# \tunit_Y) \wequi f_* \tunit_Y$
\cite[Proposition 2.4.31]{triangulated-mixed-motives} and if $f$ is étale then
$f_*(\tunit_Y) \simeq f_\#(\tunit_Y)$ \cite[Example 2.4.3(2), Definition 2.4.24
and Proposition 2.4.31]{triangulated-mixed-motives}. Let us write $\alpha_{X,Y}:
f_\#\tunit_Y \to D(f_\# \tunit_Y)$ for this canonical isomorphism. We can then
form the commutative diagram
\begin{equation*}
\begin{CD}
D(f_\# \tunit_Y) @<{\alpha_{X,Y}}<< f_\# \tunit_Y \\
@A{D(f)}AA                           @A{tr_f}AA   \\
D(\tunit_X)      @<{\alpha_{X,X}}<< \tunit_X,
\end{CD}
\end{equation*}
where $tr_f$ is defined so that the diagram commutes. This is the \emph{duality
transfer} of $f$ as defined in \cite[Section
2.3]{rigidity-in-motivic-homotopy-theory}.

Now suppose that $k$ is a perfect field. Recall that then $\SH(k)$ has a
$t$-structure. To define it, for $E \in \SH(k)$ denote by $\ul{\pi}_i(E)_j$ the
Nisnevich sheaf associated with the presheaf $X \mapsto [\Sigma^\infty X_+[i], E
\wedge \Gmp{j}]$. Then $E \in \SH(k)_{\ge 0}$ if and only if $\ul{\pi}_i(E)_j =
0$ for all $i < 0$ and all $j \in \ZZ$. This indeed defines a $t$-structure
\cite[Section 5.2]{morel-trieste}, and the its heart can be described
explicitly: it is equivalent to the category of \emph{homotopy modules}
\cite[Theorem 5.2.6]{morel-trieste}.

Let $F_* \in \SH(k)$ is a homotopy module, which we identify with an element in
the heart of the homotopy $t$-structure. Given a finite
étale morphism $f: Y \to X$ of essentially $k$-smooth schemes, write $s: X \to
Spec(k)$ for the structure map. We then define $tr_f: F_n(Y) \to F_n(X)$ as
\[ tr_f(F) := tr_f^*:
[f_\# \tunit_Y, s^* F \wedge \Gmp{n}] \to [\tunit_X, s^* F \wedge
\Gmp{n}]. \]

This transfer has the usual properties, of which we recall two.

\begin{prop}[(Base Change)] \label{prop:transfer-base-change}
Let $k$ be a perfect field, $g: V \to X$ be a morphism of
essentially $k$-smooth schemes and $f: Y \to X$ finite étale. Consider the
cartesian square
\begin{equation*}
\begin{CD}
W   @>q>> Y \\
@VpVV    @VVfV \\
V   @>>g> X.
\end{CD}
\end{equation*}
Then for any homotopy module $F_*$, we have $g^* tr_f = tr_p q^*: F_*(Y) \to
F_*(V)$.
\end{prop}
\begin{proof}
Note that $p: W \to V$ is finite étale, so this makes sense. By continuity (of
$F$), we
may assume that $X$ and $V$ are smooth (and hence so are $Y$ and $W$). Write $s:
X \to Spec(k)$ for the structure map.

If $t: A \to B$ is any map in $\SH(X)$, then the canonical diagram
\begin{equation*}
\begin{CD}
F_*(B) = [B, s^* F] @>{\circ t}>> [A, s^* F] = F_*(A) \\
@Vg^*VV                        @Vg^*VV  \\
F_*(g^* B) = [g^*B, g^*s^* F] @>{\circ g(t)}>> [g^* A, g^*s^* F] = F_*(g^* A)
\end{CD}
\end{equation*}
commutes, since $g^*$ is a
functor. Applying
this to $tr_f: \tunit_X \to f_\# \tunit_Y$ it is enough to prove that $g^*(tr_f)
= tr_p$ under the canonical identifications.

Let $f_+: f_\# \tunit_Y \wequi \Sigma^\infty_X Y_+ \to \Sigma^\infty_X X_+ =
\tunit_X$ be the canonical map (so that $tr_f = D(f_+)$ via $\alpha_{X,Y}$), and similarly for
$p_+$. Then $g^*(f_+) \wequi p_+$ and consequently $g^*(D(f_+)) \wequi D(p_+)$.
It thus remains to show that $\alpha_{\bullet,\bullet}$ is natural, i.e. that
$g^*\alpha_{X,Y} = \alpha_{V,W}: \Sigma^\infty_V W_+ \to D(\Sigma^\infty_V
W_+)$.

For this we use the notation of \cite[Example 2.4.3(2), Definition 2.4.24
and Proposition 2.4.31]{triangulated-mixed-motives}. The isomorphism
$\alpha_{X,Y}: f_\# \tunit
\to D(f_\# \tunit)$ is factored into the isomorphisms $D(f_\# \tunit) \to f_*
\tunit$, the Thom transformation $f_\# \Omega_f \tunit \to f_* \tunit$
\cite[Definition 2.4.21]{triangulated-mixed-motives} and $\Omega_f \tunit \to
\tunit$. All
of these are natural in the required sense.
\end{proof}

\begin{lemm}[(Commutation of Transfer with External Product)]
\label{lemm:comm-transfer-product}
Let $f: X' \to X$ and $g: Y' \to Y$ be finite étale. Then
\[ s_{X \times Y \#}(tr_{f \times g}) = s_{X\#}(tr_f) \wedge s_{Y\#}(tr_g):
      \Sigma^\infty (X' \times Y')_+ \wequi \Sigma^\infty X'_+ \wedge \Sigma^\infty Y'_+
      \to 
      \Sigma^\infty (X \times Y)_+ \wequi \Sigma^\infty X_+ \wedge \Sigma^\infty Y_+. \]
\end{lemm}
Here we write $s_X: X \to Spec(k)$ for the canonical map, and similarly for $Y,
X \times Y$.
\begin{proof}
Write $p_X: X
\times Y \to X$ and $p_Y: X \times Y \to Y$ for the projections.
I claim that the following diagram commutes up to natural isomorphism:
\begin{equation*}
\begin{CD}
\SH(X) \times \SH(Y) @>{p_X^* \wedge p_Y^*}>> \SH(X \times Y) \\
@V{s_{X\#} \wedge s_{Y\#}}VV                @V{s_{X \times Y\#}}VV \\
\SH(k) @= \SH(k).
\end{CD}
\end{equation*}
To prove the claim first note that there is,
for $T \in \SH(X), U \in \SH(Y)$, a natural map $s_{X \times Y
\#} (p_X^* T \wedge p_Y^* U) \to s_{X\#} T \wedge s_{y\#} U$, which can be
obtained by adjunctions, using that the pullback functors are monoidal, and that $s_{X
\times Y} = s_X \circ p_X$ (and similarly for $Y$). Then to prove that the
comparison map is an isomorphism it suffices to consider $T = \Sigma^\infty X',
U = \Sigma^\infty Y'$ for $X' \to X$ smooth any $Y' \to Y$ smooth (note that all
our functors are left adjoints and so commute with arbitrary sums, and objects
of the forms $T, U$ are generators). But then the claim boils down to
\[ X' \times_k Y' \iso (X' \times Y) \times_{X \times Y} (X \times Y') \]
which is clear.

To prove the lemma, we now specialise to $f: X' \to X$ and $g: Y' \to Y$ finite
étale. Then
\[ tr_{f \times g} = s_{X \times Y\#} (D\Sigma^\infty_{X \times Y}(f \times g)_+). \]
Note that
\[ \Sigma^\infty_{X \times Y} (f \times g)_+ = p_X^* \Sigma^\infty_X f_+ \wedge p_Y^* \Sigma^\infty_Y g_+. \]
Since $p_X^*, p_Y^*$ are monoidal we compute
\[  tr_{f \times g} = s_{X \times Y\#} p_X^* D\Sigma^\infty_X f_+ \wedge p_Y^* D\Sigma^\infty_Y g_+
             = s_{X\#} D\Sigma^\infty_X f_+ \wedge s_{Y\#} D\Sigma^\infty_Y g_+, \]
where in the last equality we have used the claim.
Since $s_{X\#} D\Sigma^\infty_X f_+ = tr_f$ by definition (and similarly for
$Y$), this is what we wanted to prove.
\end{proof}

Recall also the homotopy module $\ul{K}_*^{MW} = \ul{\pi}_0(\tunit)_*$ of Milnor-Witt
K-theory \cite[Chapter 3]{A1-alg-top}. Every homotopy module $F_*$ is a module over $\ul{K}_*^{MW}$ in the
sense that there are natural pairings $\ul{K}_*^{MW}(X) \otimes F_*(X) \to
F_{*+*}(X)$.

\begin{corr}[(Projection Formula)] \label{corr:transfer-projection-formula}
Let $k$ be a perfect field, $f: Y \to X$ a finite étale morphism of essentially
$k$-smooth schemes, and $F_*$ a homotopy module. Then for $a \in
\ul{K}_*^{MW}(Y)$ and $b \in F_*(X)$ we have $tr_f(a f^* b) = tr_f(a) b$.
Similarly for $a \in \ul{K}_*^{MW}(X)$ and $b \in F_*(Y)$ we have $tr_f(f^*(a)
b) = a tr_f(b)$.
\end{corr}
\begin{proof}
The usual proof works, see for example \cite[Proof of Corollary
3.4]{calmes2014finite}. We review it. We only show the first statement, the
second is similar. Consider the cartesian square
\begin{equation*}
\begin{CD}
Y @>{(\id \times f) \delta_Y}>> Y \times X \\
@VfVV                           @VV{f \times \id}V \\
X @>>{\delta_X}> X \times X,
\end{CD}
\end{equation*}
where $\delta_X: X \to X \times X$ is the diagonal and similarly for $Y$.
We have the map $\beta: \Sigma^\infty Y_+ \wedge \Sigma^\infty X_+ \to
\ul{K}_*^{MW} \wedge F \to F$, where $\ul{K}_*^{MW} \wedge F \to F$ is the
module structure and the first map is the tensor product of $\Sigma^\infty Y_+
\to \ul{K}_*^{MW}$ (corresponding to $a$) and $\Sigma^\infty X_+ \to F$
(corresponding to $b$). This defines an element $\beta \in F(Y \times X)$. We
have $tr_f ((\id \times f) \delta_Y)^* \beta = tr_f(a f^* b)$ and
$\delta_X^* tr_{f \times \id_Y} \beta = tr_f(a) b$
(the latter since $tr_{\id} = \id$ and
$tr_{f \times g} (x \otimes y) = tr_f(x) \otimes tr_g(y)$ by Lemma
\ref{lemm:comm-transfer-product}). These two elements
are equal by the base change formula, i.e. Proposition
\ref{prop:transfer-base-change}.
\end{proof}

\section{Recollections on Pre-Motivic Categories}
\label{sec:recoll-pre-motivic}

The six functors formalism \cite[Section A.5]{triangulated-mixed-motives} is a
very strong, and very general, duality theory. As such it is no surprise that
proving that any theory satisfies it requires some work. Fortunately it is now
possible to reduce this to establishing a few axioms.

Let $\mathcal{S}$ be a base category of schemes. Recall that a \emph{pre-motivic
category} $\mathcal{M}$ over $\mathcal{S}$ consists of
\cite[Definition A.1.1]{cisinski2013etale} a pseudofunctor
$\mathcal{M}$ on $\mathcal{S}$, taking values in triangulated, closed symmetric
monoidal
categories. Often these categories will be obtained as the homotopy categories
of a pseudofunctor taking values in suitable Quillen model categories and left
Quillen functors. For $f: X \to Y \in \mathcal{S}$, the functor $\mathcal{M}(f):
\mathcal{M}(Y) \to \mathcal{M}(X)$ is denoted $f^*$. For any $f$, the functor
$f^*$ has a triangulated right adjoint $f_*$ (which is not required to be
monoidal). If $f$ is smooth, then $f^*$ has a triangulated left adjoint $f_\#$
(also not required to be monoidal). Moreover, $\mathcal{M}$ needs to satisfy
smooth base change and the smooth projection formula, in the following sense.

Let
\begin{equation*}
\begin{CD}
Y @>q>> X \\
@VgVV   @VVfV \\
T @>p>> S
\end{CD}
\end{equation*}
be a cartesian square in $\mathcal{S}$, with $p$ smooth. Then smooth base change
means that the natural
transformation $q_\# g^* \to f^* p_\#$ is required to be a natural isomorphism.

Finally, let $f: Y \to X$ be a smooth morphism in $\mathcal{S}$. Then the smooth
projection formula means that, for $E \in \mathcal{M}(X)$ and $F \in
\mathcal{M}(Y)$ we have $f_\#(F \otimes f^* E) \wequi f_\#(F) \otimes E$, via
the canonical map.

Here are some further properties a pre-motivic category can satisfy. We say
$\mathcal{M}$ satisfies the \emph{homotopy property} if for every $X \in
\mathcal{S}$ the natural map $p_\# \tunit \to \tunit \in \mathcal{M}(X)$ is an
isomorphism, where $p: \Aone \times X \to X$ is the canonical map.

Let now $q: \mathbb{P}^1 \times X \to X$ be the canonical map. We say that
$\mathcal{M}$ satisfies the \emph{stability property} if the cone of the
canonical map $q_\# \tunit \to \tunit \in \mathcal{M}(X)$ is a
$\otimes$-invertible object. In this case we write $\tunit(1) = fib(q_\# \tunit
\to \tunit)[-2]$ and then as usual $E(n) = E \otimes \tunit(1)^{\otimes n}$ for
$n \in \ZZ, E \in \mathcal{M}(X)$.

Finally, let $X \in \mathcal{S}$, $j: U \to X \in \mathcal{S}$ an open
immersion, and $i: Z \to \mathcal{S}$ a complementary closed immersion. Then
for $E \in \mathcal{M}(U)$ there are the adjunction maps
\[ j_\#j^* E \to E \to i_* i^* E. \]
We say that $\mathcal{M}$ satisfies the \emph{localisation property} if these
maps are always part of a distinguished triangle.

One then has the following fundamental result. It was discovered by
Voevodsky, first worked out in detail by Ayoub, and then formalised by
Cisinski-Déglise.

\begin{thm}[(Ayoub, Cisinski-D\'eglise)] \label{thm:fundamental-six-functors}
Let $\mathcal{S}$ be the category of Noetherian schemes of finite dimension and
$\mathcal{M}$ a pre-motivic category which satisfies the homotopy property, the
stability property, and the localisation property. Then if $\mathcal{M}(X)$ is a
well-generated triangulated category for every $X$, $\mathcal{M}$ satisfies
the full six functors formalism.
\end{thm}
\begin{proof}
This is proved for ``adequate categories of schemes'' in
\cite[Theorem 2.4.50]{triangulated-mixed-motives}, of which Noetherian finite
dimensional schemes are an example.
\end{proof}

One further property we will make use of is \emph{continuity}. This can be
formulated as follows. Let $\{S_\alpha\}_{\alpha \in A}$ be an inverse system in $\mathcal{S}$,
where all the transition morphisms are affine and the limit $S := \lim_\alpha
S_\alpha$ exists in $\mathcal{S}$. Write $p_\alpha: S \to S_\alpha$ for the
canonical projection. Let $E \in \mathcal{M}(S_{\alpha_0})$ for some $\alpha_0
\in A$ and write for $\alpha > \alpha_0,$ $E_\alpha = (S_\alpha \to
S_{\alpha_0})^* E$. We say that $\mathcal{M}$ satisfies the continuity property
if for every affine inverse system $S_\alpha$ as above, every $E$ and every $i
\in \ZZ$ the canonical map
\[ \colim_{\alpha > \alpha_0} [\tunit(i), E_\alpha]_{\mathcal{M}(S_\alpha)}
    \to [\tunit(i), p_{\alpha_0}^* E]_{\mathcal{M}(S)} \]
is an isomorphism.

We in particular use the following consequence of continuity and localisation.

\begin{corr} \label{corr:conservative-detection}
Suppose that $\mathcal{M}$ be a pre-motivic category over $\mathcal{S}$ (where
$\mathcal{S}$ contains all Henselizations of its schemes), coming from a
pseudofunctor valued in model categories. Assume that
$\mathcal{M}$ satisfies continuity and localisation.

Let $E \in \mathcal{M}(X)$, where $X$ is Noetherian of finite dimension.
Then $E \wequi 0$ if and only if for every morphism $f: Spec(k) \to X$ with $k$
a field we have $f^*E \wequi 0$.
\end{corr}
\begin{proof}
By localisation, we may assume that $X$ is reduced (see for example
\cite[Proposition 2.3.6(1)]{triangulated-mixed-motives}). By \cite[Proposition
4.3.9]{triangulated-mixed-motives} (this result requires $\mathcal{M}$ to come
from a model category) we may assume that $X$ is (Henselian) local
with closed point $x$ and open complement $U$.
By localisation, it suffices to show that $E|_x \wequi 0$ and $E|_U \wequi 0$.
The former holds by assumption, and the latter by induction on the dimension. This
concludes the proof.
\end{proof}

\paragraph{Example.} The pseudofunctors $X \mapsto \SH(X)$ and $X \mapsto
D_\Aone(X)$ satisfy the six functors formalism and continuity (for the base
category of Noetherian finite dimensional schemes)
\cite{ayoub2007six,triangulated-mixed-motives}.

\section{Recollections on Monoidal Bousfield Localization}
\label{sec:monoidal-local}

Let $\mathcal{M}$ be a monoidal model category and $\alpha: Y' \to Y \in
\mathcal{M}$ a morphism. We wish to ``monoidally invert $\alpha$'', by which we
mean passing to a model category $L_\alpha^\otimes \mathcal{M}$ obtained by
localizing $\mathcal{M}$ and such that for every $T \in L_\alpha^\otimes
\mathcal{M}$ the induced map $\alpha_T: T \otimes^L Y' \to T \otimes^L Y$ is a
weak equivalence. We will also write $L_\alpha^\otimes \mathcal{M} =:
\mathcal{M}[\alpha^{-1}]$ and even $Ho(\mathcal{M}[\alpha^{-1}]) =:
Ho(\mathcal{M})[\alpha^{-1}]$.

The monoidal $\alpha$-localisation exists very generally. Suppose
that $Y'$ and $Y$ are cofibrant, and that $\mathcal{M}$ admits a
set of cofibrant homotopy generators $G$
(for example $\mathcal{M}$
combinatorial \cite[Corollary 4.33]{barwick2010left}). Let $H_\alpha = \{Y'
\otimes T \xrightarrow{\alpha \otimes \id} Y \otimes T| T \in G\}$. When no
confusion can arise, we will denote $H_\alpha$ just by $H$. Then the
Bousfield localisation $L_H \mathcal{M}$, if it exists (for example if
$\mathcal{M}$ is left proper and combinatorial) is $\mathcal{M}[\alpha^{-1}]$. We will call
$H_\alpha$-local objects $\alpha$-local. As a further sanity
check, the model category $L_H \mathcal{M}$ is still monoidal as follows from
\cite[Proposition 4.47]{barwick2010left}.

The situation simplifies somewhat if $Y'$ and $Y$ are invertible and
$\mathcal{M}$ is stable. Then we may as
well assume that $Y' = \tunit$. Given $T \in \mathcal{M}$ cofibrant we can consider the
directed system
\[ T \iso T \otimes \tunit \xrightarrow{\id \otimes \alpha} T \otimes Y \iso T
   \otimes Y \otimes \tunit \to T \otimes Y^{\otimes 2} \to \dots \]
and its homotopy colimit $T[\alpha^{-1}] := \hocolim_n T \otimes X^{\otimes n}$.
More generally, if $T$ is not cofibrant, we can either first cofibrantly replace
it, or use the derived tensor product. Either way, we denote the result still by
$T[\alpha^{-1}]$. The main point of this section is to show that under suitable
conditions, $T[\alpha^{-1}]$ is the $\alpha$-localization of $T$.

Clearly this is only a reasonable expectation under some compact generation
assumption. More generally, one would expect a transfinite iteration of
$\alpha$. Since all our applications will be in compactly generated situations,
we refrain from giving the more general argument.

Recall that by a set of compact homotopy generators $G$ for $\mathcal{M}$ we
mean a set of (usually cofibrant) objects $G \subset Ob(\mathcal{M})$ such that
$\mathcal{M}$ is generated by the objects in $G$ under homotopy colimits, and
such that
for any directed system $X_1 \to X_2 \to \dots \in \mathcal{M}$ and $T \in G$, the canonical
map $\hocolim_i \Map^d(T, X_i) \to \Map^d(T, \hocolim_i X_i)$ is an equivalence.

\begin{lemm} \label{lemm:rho-local-model}
Let $\alpha: \tunit \to Y$ be a map between objects in a symmetric
monoidal, stable model category such that $Y$ is invertible (in the homotopy
category).  Assume that $\mathcal{M}$
has a set of compact homotopy generators $G$, and that $\mathcal{M}[\alpha^{-1}]$
exists.

Then for each $U \in
\mathcal{M}$ the object $U[\alpha^{-1}]$ is $\alpha$-local and
$\alpha$-locally weakly equivalent to $U$. In other words, $U \mapsto
U[\alpha^{-1}]$ is an $\alpha$-localization functor.

Also $G$ defines a set of compact homotopy generators for
$\mathcal{M}[\alpha^{-1}]$.
\end{lemm}
\begin{proof}
We first show that the images of $G$ in $Ho(\mathcal{M}[\alpha^{-1}])$ are compact homotopy
generators. Generation is clear, and for homotopy compactness it is enough to show that a
filtered homotopy colimit of $\alpha$-local objects is $\alpha$-local. But this
follows from homotopy compactness of $T \otimes Y^{\otimes n}$ (for $T \in G$ and $n \in
\{0, 1\}$) and definition of $\alpha$-locality.

In a model category
$\mathcal{N}$ with compact homotopy generators, if $T_1 \to T_2 \to \dots$ is a directed system of weak
equivalences then $\hocolim_i T_i$ is weakly equivalent to $T_1$. (This follows
from the same result in the category of simplicial sets.) Thus
$U[\alpha^{-1}]$ is $\alpha$-locally weakly equivalent to $U$.

It remains to see that $U[\alpha^{-1}]$ is $\alpha$-local. This follows from the
next two lemmas.
\end{proof}

In the above lemma, we have defined an object $X$ to be $\alpha$-local if for
all $T \in \mathcal{M}$ the induced map $\alpha^*: \Map^d(T \otimes^L Y, X) \to
\Map^d(T, X)$ is an equivalence, because
this is the way Bousfield localization works. Another
intuitively appealing property would be for the canonical map $X \to X \otimes Y$
to be an equivalence. As the next lemma shows, these two notions agree in our
case.

\begin{lemm}
Let $\mathcal{M}$ be a symmetric monoidal model category and $\alpha: \tunit \to Y$ a morphism with
$Y$ invertible.

Call an object $X \in \mathcal{M}$ $\alpha'$-local if $X \to X \otimes^L Y$ is a
weak equivalence.
Then $X$ is $\alpha$-local if and only if $X$ is $\alpha'$-local, if and only if
$X$ is $\alpha \otimes \alpha$-local.
\end{lemm}
\begin{proof}
We shall show that (1) $X$ is $\alpha$-local if and only if it is $\alpha
\otimes \alpha$-local, (2) $X$ is $\alpha'$-local if and only if it is $(\alpha
\otimes \alpha)'$-local, (3) $X$ is $\alpha'$-local if it is $\alpha$-local and
(4) $X$ is $\alpha \otimes \alpha$-local if it is $(\alpha \otimes
\alpha)'$-local.

All tensor products and mapping spaces will be derived in this proof.

(1)
Consider the string of maps
\[ \Map(T \otimes Y^{\otimes 3}, X) \to \Map(T \otimes Y^{\otimes 2}, X)
   \to \Map(T \otimes Y, X) \to \Map(T , X). \]
If $X$ is $\alpha \otimes \alpha$-local, then the composite of any two
consecutive maps is an equivalence, and hence all maps are equivalences by
2-out-of-6. Consequently $X$ is $\alpha$-local. The converse is clear.

(2)
Consider the string of maps
\[ X \to X \otimes Y \to X \otimes Y^{\otimes 2} \to X \otimes Y^{\otimes 3}. \]
If $X$ is $(\alpha \otimes \alpha)'$-local then so is $Z \otimes X$ for any $Z$,
since (derived) tensor product preserves weak equivalences. It follows that $X
\otimes Y$ is $(\alpha \otimes \alpha)'$-local, and hence the composite of any two
consecutive maps is an equivalence. Again by 2-out-of-6 this implies that $X$ is
$\alpha'$-local. The converse is clear.

(3)
An object $X$ is $\alpha$-local if (and only if) for all $T \in \mathcal{M}$ the map $\Map(T
\otimes Y, X) \to \Map(T, X)$ is a weak equivalence (of simplicial sets). In
particular $T \to T \otimes Y$ is an $\alpha$-local weak equivalence for all
$T$. It also
follows that $X \otimes Y$ is $\alpha$-local if $X$ is (here we use
invertibility of $Y$). Since $X \to X \otimes
Y$ is an $\alpha$-local weak equivalence, it is a weak equivalence if $X$ (and
hence $X \otimes Y$) is $\alpha$-local. Thus $X$ is $\alpha'$-local if it is
$\alpha$-local.

(4)
For any simplicial set $K$ we have $[K, \Map(T, X)] = [K \otimes T, X]$ (using
a framing if the model category is not simplicial).
It follows that $X$ is $\alpha$-local
if and only if for all $T \in \mathcal{M}$ the map $\alpha^*: [T \otimes Y, X] \to [T, X]$ is an
isomorphism. In particular, this property can be checked entirely in the
homotopy category of $\mathcal{M}$, in which we will work from now on.

Suppose, for now, that $X$ is $\alpha'$-local. (We will find that
our strategy does not work, but it will work for $\alpha \otimes \alpha$, and
this is all that is left to prove.) We can choose an inverse equivalence $\beta:
X \otimes Y \to X$. We consider the map $\overline{\beta}: [T, X] \to [T \otimes
Y, X]$ sending $f: T \to X$ to $T \otimes Y \xrightarrow{f \otimes \id} X
\otimes Y \xrightarrow{\beta} X$. We would like to say that $\overline{\beta}$
is inverse to $\alpha^*$. Given $f: T \to X$ we get a commutative diagram
\begin{equation*}
\begin{CD}
T \otimes Y @>{f \otimes \id}>> X \otimes Y   \\
@A{\alpha}AA                     @A{\alpha}AA \\
T           @>f>>               X.
\end{CD}
\end{equation*}
Consequently $\alpha_* \alpha^* \overline{\beta} = \alpha_*: [T, X] \to [T, X
\otimes Y]$ and thus $\alpha^* \overline{\beta}  = \id$ (note that $\alpha_*$
means composition with $X \to X \otimes Y$, which is an isomorphism).

The problem is with showing that $\overline{\beta} \alpha^* = \id$. For this we
fix $f: T \otimes Y \to X$ and consider the diagram
\begin{equation*}
\begin{CD}
T \otimes Y \otimes Y @>{f \otimes \id}>> X \otimes Y  \\
@A{\id \otimes \alpha \otimes \id}AA       @A{\alpha}AA  \\
T \otimes \tunit \otimes Y @>f>>          X.
\end{CD}
\end{equation*}
If it commutes for all such $f$, then $\overline{\beta} \alpha^* = \id$. But
this is not clear; the two paths differ by a switch of $Y$.

However, in any symmetric monoidal category, the switch isomorphism on the
\emph{square} of an invertible object is the identity \cite[Propositions 4.20
and 4.21]{dugger2014coherence}. Consequently our argument works for $\alpha
\otimes \alpha$, and this is what we set out to prove.
\end{proof}

\paragraph{Remark.} The assumption that $Y$ is invertible is necessary in
general for the
above result. For example, if $\mathcal{M}$ is a \emph{cartesian} symmetric
monoidal model category, then there cannot be any $\alpha'$-local objects unless
$* = \tunit \to Y$ is already an equivalence.

\begin{lemm}
Notations and assumptions as in Lemma \ref{lemm:rho-local-model}.

For any (cofibrant) $X \in \mathcal{M}$, the object $X[\alpha^{-1}]$ is
$\alpha$-local.
\end{lemm}
\begin{proof}
By the previous lemma, it suffices to show that $X[\alpha^{-1}]$ is $(\alpha
\otimes \alpha)'$-local. Clearly $X[\alpha^{-1}] \simeq X[(\alpha \otimes
\alpha)^{-1}]$, i.e. we may assume without loss of generality that $Y$ is a
square, and so its switch isomorphism (in the homotopy category) is the identity.

Since tensor product commutes with colimits (in each variable) we have  $X[1/f]
\simeq X \otimes^L \tunit[1/f]$, and we can simplify notation by assuming without
loss of generality that $X = \tunit$.

What we need to prove is that the following diagram induces an equivalence on
homotopy colimits:
\begin{equation*}
\begin{CD}
\tunit @>f_1>>     G @>f_2>>          G \otimes G @>f_3>>          G \otimes G \otimes G @>>> \dots \\
@Vh_1VV                @Vh_2VV                          @Vh_3VV                  @Vh_4VV \\
G @>f'_2>> G \otimes G @>f'_3>> G \otimes G \otimes G @>f'_4>> G \otimes G \otimes G @>>> \dots \\
\end{CD}
\end{equation*}

Because of the domains and codomains, it is tempting to guess that $f_i \simeq
h_i \simeq f'_i$. Here we write $f \simeq g$ to mean that the maps become equal
in the homotopy category. We claim that this guess is correct. Then if
$T$ is any homotopy compact object, applying $[T, \bullet]$ to our diagram we get
a diagram of abelian groups which we need to show induces an isomorphism on
colimits. Homotopic maps become equal when applying $[T, \bullet]$, and then
the desired result follows from an easy diagram chase. By compact generation and
stability, this will conclude the proof.

It remains to prove the claim. For this we may work entirely in the homotopy
category, which we will do from now on.
It is easy to see that indeed $f_i = h_i$. For general $Y$, it would not be true
that $f'_i = f_i$; one may check that the maps differ by appropriate
switches of $Y$. However, we have assumed that the switch on $Y$ \emph{is} the
identity, so indeed $f_i = f'_i$ as well.
\end{proof}

\paragraph{Remark.} The stability assumption was used in the above proof in
the following form: if $A \to B$ is any morphism in $\mathcal{M}$ and $[T, A]
\to [T, B]$ is an isomorphism for all homotopy compact $T$, then $A \to B$ is a weak
equivalence.
This fails for example in the homotopy category of spaces.

The stability assumption is in fact necessary for the above result.
The author learned the following counterexample from Marc Hoyois:
let $\mathcal{M}$ be the model category of small, stable
$\infty$-categories, $Y = \tunit$ the category of finite spectra and $\alpha
= 2$, i.e. the functor which sends a finite spectrum $s$ to $s \oplus s$. Then
$\mathcal{C} \in \mathcal{M}$ is $\alpha'$-local only if it is trivial.
Indeed for $c \in \mathcal{C}$ the map $[c, c] \to [c \oplus c, c
\oplus c]$ needs to be an isomorphism, which forces $c \simeq 0$. But one may
show that $\tunit[1/\alpha]$ is not the zero category, and so is not
$\alpha'$-local (let alone $\alpha$-local).

See \cite[Theorem 3.7]{hoyois2016equivariant} for a criterion that can be
applied in unstable situations.

\section{The Theorem of Jacobson and $\rho$-stable Homotopy Modules}
\label{sec:jacobson}
Throughout this section, $k$ is a field of characteristic zero. Recall that the
real étale topology is finer than the Nisnevich topology; in particular every
real étale sheaf is a Nisnevich sheaf.

\begin{thm}[(Jacobson \cite{jacobson-fundamental-ideal}, Theorem 8.5)]
There is a canonical isomorphism (in $Shv_{Nis}(Sm(k))$)
\[ \colim_n \ul{I}^n \to a_{\ret} \ul{\ZZ}, \]
where the transition maps $\ul{I}^n \to \ul{I}^{n+1}$ are given by
multiplication with $2 = \langle 1, 1 \rangle \in \ul{I}$.
\end{thm}
Here $\ul{I}$ denotes the sheaf of fundamental ideals on $Sm(k)_{Nis}$, i.e. the
sheaf associated with the presheaf $X \mapsto I(X)$, where $I(X)$ is the
fundamental ideal of the Witt ring of $X$ \cite{knebusch-bilinear}. We
similarly write $\ul{W}$ for the sheaf of Witt rings, etc.

Let us recall the construction of the isomorphism in Jacobson's theorem.
If $\phi \in W(K)$, where $K$ is a field,
and $p$ is an ordering of
$K$, then there is the signature $\sigma_p(\phi) \in \ZZ$. If $\phi \in W(X)$,
define $\sigma(\phi): R(X) \to \ZZ$ as follows. For $(x, p) \in R(X)$ put
$\sigma(\phi)(x, p) = \sigma_p(\phi|_x)$. Then one shows that $\sigma(\phi)$ is
a \emph{continuous} function from $R(X)$ to $\ZZ$, i.e. an element of
$H^0_{\ret}(X, \ZZ)$.

Next if $\phi \in I(k)$ then $\sigma_p(\phi) \in 2\ZZ$. Consequently if $\phi
\in I(X)$ also $\sigma(\phi) \in 2H^0_{\ret}(X, \ZZ)$. We may thus define
$\tilde{\sigma}(\phi) = \sigma(\phi)/2$ and in this way we obtain
$\tilde{\sigma}: I(X) \to H^0_{\ret}(X, \ZZ)$. Similarly we get $\tilde{\sigma}:
I^n(X) \to H^0_{\ret}(X, \ZZ)$ with $\tilde{\sigma}(\phi) = \sigma(\phi)/2^n$ for
$\phi \in I^n(X)$. For each $n$ there is a commutative diagram
\begin{equation*}
\begin{CD}
I^n(X) @>{\tilde{\sigma}}>> H^0_{\ret}(X, \ZZ) \\
@V2VV                         @|              \\
I^{n+1}(X) @>{\tilde{\sigma}}>> H^0_{\ret}(X, \ZZ). \\
\end{CD}
\end{equation*}
Consequently there is an induced map $\tilde{\sigma}: \colim_n I^n(X) \to
H^0_{\ret}(X, \ZZ)$. The claim is that this is an isomorphism after sheafifying,
i.e. for $X$ local.

\begin{corr} \label{corr:colim_KnMW}
Let $\ul{K}_n^{MW}$ denote the $n$-th unramified Milnor-Witt K-theory sheaf.
Then there is a canonical isomorphism $\colim_n \ul{K}_n^{MW} \to a_{\ret} \ZZ$.
Here the colimit is along multiplication with $\rho := -[-1] \in K_1^{MW}(k)$.
\end{corr}
\begin{proof}
Recall the element $h \in K_0^{MW}(k)$ with the following properties:
$\ul{K}_n^{MW}/h = \ul{I}^n$ \cite[Theoreme 2.1]{morel2004puissances}
and for $a \in K_1^{MW}(k)$ we have
$a^2 h = 0$ \cite[Corollary 3.8]{A1-alg-top}
(this relation is the analogue of the fact that in a
graded commutative ring $R_*$ with $a \in R_1$ we have $a^2 = -a^2$ by graded
commutativity, so $2a^2 = 0$). Consequently $\rho^2 h = 0$ and so $\colim_n
\ul{K}_n^{MW} \to \colim_n \ul{I}^n$ is an isomorphism. It remains to note that
the image of $\rho$ in $K_1^{MW}/h(k) \iso I(k)$ is given by $-(\langle -1
\rangle -1) = 2 \in I(k) \subset W(k)$, so the induced transition maps in the
colimit are precisely those used in Jacobson's theorem.
\end{proof}

Note that the sheaves $\ul{I}^n$ form a homotopy module, namely the homotopy
module of Witt $K$-theory \cite[Examples 3.33 and 3.33]{A1-alg-top}
\cite[Theoreme 2.1]{morel2004puissances}; see also \cite{gille2016milnor}.
Consequently they have transfers for
finite separable field extensions. The sheaf $a_{\ret} \ul{\ZZ}$ also has
transfers for finite (separable) field extensions. Indeed if $l/k$ is finite
then $Sper(l) \to Sper(k)$ is a finite-sheeted local homeomorphism
\cite[3.5.6 Remark (ii)]{scharlau2012quadratic} and hence we transfer by
``taking sums over the values at the preimages''.

\begin{lemm}
The isomorphism $\colim_n \ul{I}^n \to a_{\ret} \ul{\ZZ}$ is compatible with
transfers on fields.
\end{lemm}
\begin{proof}
It suffices to prove that for a field $k$, the total signature $W(k) \to
H^0_{\ret}(k, \ZZ)$ is compatible with transfer. Let $l/k$ be a finite extension
and $R/k$ a real closure. There is a commutative diagram
\begin{equation*}
\begin{CD}
W(l) @>>> W(R \otimes_k l) \\
@V{tr}VV   @V{tr}VV        \\
W(k) @>>> W(R)
\end{CD}
\end{equation*}
by the base change formula, i.e. Proposition \ref{prop:transfer-base-change}.
Note that $Sper(R \otimes_k l)$ is the fibre of $Sper(l) \to
Sper(k)$ over the ordering corresponding to the inclusion $k \subset R$.
Consequently we also have the commutative diagram
\begin{equation*}
\begin{CD}
H^0_{\ret}(l, \ZZ) @>>> H^0_{\ret}(R \otimes_k l, \ZZ) \\
@V{tr}VV   @V{tr}VV        \\
H^0_{\ret}(k, \ZZ) @>>> H^0_{\ret}(R, \ZZ).
\end{CD}
\end{equation*}
Since the signature maps are determined by pulling back to a real closure, this
means that we may assume that $k$ is real closed. (Since both sides we are
trying to prove equal are additive, we may still assume that $l$ is a field.)
But then either $l = k$ or $l = k[\sqrt{-1}]$. In the former case the transfer
on both sides is the identity, and in the latter it is zero.
\end{proof}

We will make good use of the following observation.

\begin{corr} \label{corr:transfer-aretZZ-surjective}
Let $l_1, \dots, l_r/k$ be finite extensions such that $\coprod_i Spec(l_i) \to
Spec(k)$ is a rét-cover. Then $tr: \oplus_i H^0_{\ret}(l_i, \ZZ) \to H^0_{\ret}(k,
\ZZ)$ is surjective.
\end{corr}
\begin{proof}
The map $\coprod_i Sper(l_i) \to Sper(k)$ is a surjective local homeomorphism of
compact, Hausdorff, totally disconnected spaces \cite[Theorem
3.5.1 and Remarks 3.5.6]{scharlau2012quadratic}. The result thus follows from the next lemma.
\end{proof}

\begin{lemm}
Let $\phi: X \to Y$ be a surjective local homeomorphism of compact, Hausdorff, totally
disconnected spaces. Then $\phi$ has finite fibers, and the ``summing over preimages'' transfer
$H^0(X, \ZZ) \to H^0(Y, \ZZ)$ is surjective.
\end{lemm}
\begin{proof}
The claim that $\phi$ has finite fibers is well-known. We include a proof for
convenience of the reader: since $\phi$ is a local homeomorphism the fibers are
discrete, since $Y$ is Hausdorff they are closed, and since $X$ is compact they
are compact. Now observe that a compact discrete space is finite.

We now prove the surjectivity of the transfer. First we make the following claim: if
$X$ is a compact, Hausdorff, totally disconnected space, then given $x \in U
\subset X$ with $U$ open, there exists $x \in V \subset U$ such that $V$ is
clopen in $X$. Indeed for $y \ne x$ let $U_y$ be a clopen neighbourhood of $y$
disjoint from $x$. Then $\cup_{y \in X \setminus U} U_y$ is an open cover of the
compact (since closed) complement $X \setminus U$. Let $U_1, \dots, U_n$ be a
finite subcover. Then $V = X \setminus \cup_i U_i$ works.

Now consider the morphism $\phi: X \to Y$. For $y \in
Y$ choose a clopen neighbourhood $U_y$ of $y \in
Y$ such that there exists a clopen set $V_y \subset X$ with $\phi(V_y) = U_y$
and $\phi: V_y \to U_y$ a homeomorphism. We will say in this situation that $\phi$ \emph{splits
strongly over $U_y$}. We note that such $V_y, U_y$ exist: since $\phi$ is a local
homeomorphism, there exists $V_y' \subset X$ such that $U_y' := \phi(V_y)$ is an
open neighbourhood of $y$ and $\phi: V_y' \to U_y'$ is a homeomorphism. By the
claim, we may assume that $V_y'$ is clopen. Now
choose a clopen neighbourhood $U_y \subset U_y'$, using the claim again. Then $V_y :=
\phi^{-1}(U_y) \cap V_y'$ is clopen in $X$ and maps homeomorphically to $U_y$.

We obtain in this way an open cover
$\{U_y\}_{y \in Y}$ of $Y$. Since $Y$ is compact,
we can choose a finite subcover $U_1, \dots,
U_n$. Using that all the $U_i$ are clopen we can refine further until we have
found a disjoint clopen cover (replace $U_i$ by $U_i \setminus (U_1 \cup
U_2 \cup \dots \cup U_{i-1})$) over which $\phi$ splits strongly. (Note that if
$\phi$ splits strongly over a clopen $U \subset Y$, then it also splits strongly
over any clopen $U' \subset U$.)

Since $H^0(Y, \ZZ)$ is the set of continuous functions from $Y$ to
$\ZZ$, it suffices to prove that the indicator function $\chi_{U_i}:
Y \to \ZZ$ of the clopen subset $U_i$ is in the image of transfer (because
$1 = \sum_i \chi_{U_i}$, and the transfer is additive). But
$\phi$ is strongly split over $U_i$ by construction, so there exists some clopen subset $U
\subset X$ such that $\phi: U \to U_i$ is a homeomorphism.
Then $\chi_{U} \in H^0(X, \ZZ)$ and this is taken by transfer to
$\chi_{U_i}$, as follows from the explicit description of transfer in terms of
``summing over preimages''.
\end{proof}

We will want to show that certain presheaves are sheaves in the rét-topology. We
find it easiest to first develop a criterion for this. We start with the
following result, which is surely well known.

\begin{lemm} \label{lemm:sheaf-cond-technical}
Let $\tau$ be a topology on a category $\mathcal{C}$ and $F$ a presheaf on
$\mathcal{C}$ which is $\tau$-separated. Let $X \in \mathcal{C}$ and $U_\bullet,
V_\bullet \to X$ be $\tau$-coverings. Suppose that $V_\bullet$ refines
$U_\bullet$, i.e. we are given a morphism $f: V_\bullet \to U_\bullet$ over $X$.
Then if $F$ satisfies the sheaf condition with respect to $V_\bullet$, it also
satisfies the sheaf condition with respect to $U_\bullet$.
\end{lemm}
\begin{proof}
The proof can be extracted from the proof of \cite[Tag 00VX]{stacks-project}. We
repeat the argument for convenience. For simplicity, suppose that $U_\bullet$
and $V_\bullet$ use the same indexing set $I$, and that the refinement is of the
form $V_i \to U_i$. We are given $s_i \in F(U_i)$ for each $i$, such that
$s_i|_{U_i \times_X U_j} = s_j|_{U_i \times_X U_j}$, and we need to show that
there is a (necessarily unique) $s \in F(X)$ with $s|_{U_i} = s_i$.

Let $t_i = f^* s_i$. Then $t_\bullet$ is a compatible family for the covering
$V_\bullet$, and hence there is $s \in F(X)$ with $s|_{V_i} = t_i$
for all $i$. We need to show that also $s|_{U_i} = s_i$. For this, fix $i_0 \in
I$ and consider the coverings $U_\bullet', V_\bullet' \to U_{i_0}$ obtained by
base change. Then $V_\bullet'$ refines $U_\bullet'$. We find that $s_{i_0}|_{U_{i_0}
\times_X U_i} = s_i|_{U_{i_0} \times_X U_i}$ by assumption, and hence
$f^*(s_{i_0}|_{U_{i_0} \times_X U_i}) = f^*(s_i|_{U_{i_0} \times_X U_i}) =
t_i|_{U_{i_0} \times_X V_i} = s|_{U_{i_0} \times_X V_i}$ by construction.
But now because $F$ is
separated in the $\tau$-topology and $V_\bullet' \to U_{i_0}$ is a cover we
conclude that $s_{i_0} = s|_{U_{i_0}}$, as needed.
\end{proof}

\begin{corr} \label{corr:ret-criterion}
Let $F$ be a sheaf on $Sm(k)_{Nis}$. Then $F$ is a sheaf in the rét-topology if
and only if $F$ satisfies the sheaf condition for every rét-cover $f\colon U \to X$,
where $X$ is (essentially) smooth, Henselian local and $f$ is \emph{finite} étale.
\end{corr}
\begin{proof}
For this proof, we call a morphism with the properties of $f$ a frét-cover.

The condition is clearly necessary; we show the converse.

(*) We first claim that every rét-cover $U_\bullet \to X$ with $X$ smooth Henselian
local can be refined by a frét-cover. We can
certainly refine $U_\bullet$ by an affine cover, so assume that each $U_i$ is
affine. Then by \cite[Tag
04GJ]{stacks-project} each $U_i$ splits as $U_i' \coprod U_i''$ with $U_i' \to
X$ finite étale and $U_i'' \to X$ not hitting the closed point $m$ of $X$ (note that $U_i
\to X$ is everywhere quasi-finite). I claim that
$U_\bullet'$ is also a rét-cover. Indeed étale morphisms induce open maps on
real spectra \cite[Proposition 1.8]{real-and-etale-cohomology} and $U_\bullet'$
covers $R(m) \subset R(X)$ by construction. But the only open subset of $R(X)$ containing
$R(m)$ is all of $R(X)$, by \cite[Propositions II.2.1 and
II.2.4]{andradas2012constructible}. Finally the real spectrum of any ring is
quasi-compact \cite[II.1.5]{andradas2012constructible} whence we can always
refine by a finite subcover, and then taking the disjoint union we refine by a
singleton cover.

If $X \in Sm(k)$, we write $F_X := F|_{X_{Nis}}$ for the restriction to the
small site. Write
$\iHom$ for the internal mapping presheaf functor in this category. Recall that
$\iHom(V, F_X)(V') = F(V \times_X V')$; in particular this functor
preserves sheaves.

Let $U_\bullet \to X$ be a rét-cover. To show that
$F(X) \to F(U_\bullet) \rightrightarrows F(U_\bullet \times_X U_\bullet)$ is an
equaliser diagram (respectively the first map is injective),
it is sufficient to show that $F_X \to \iHom(U_\bullet, F_X) \rightrightarrows
\iHom(U_\bullet \times_X U_\bullet, F_X)$ is an equaliser diagram of sheaves (respectively
the first map is an injection of sheaves), since limits of sheaves are computed
in presheaves.
But finite limits (respectively injectivity) are detected on stalks,
whence in both situations we may assume that $X$ is Henselian local. (**)

Now we show that $F$ is rét-separated. Let $U_\bullet \to X$ be a rét-cover.
By the above, to show that $F(X) \to F(U_\bullet)$ is injective we may assume that $X$
is Henselian local. Then $U_\bullet \to X$ is refined by a
frét-cover $V \to X$, by (*). But then $F(X) \to F(U_\bullet) \to F(V)$ is injective since $F$
satisfies the sheaf condition for $V \to X$ by assumption, so $F(X) \to F(U_\bullet)$ is
injective and $F$ is rét-separated.

Finally let $U_\bullet \to X$ be any rét-cover. We wish to show that $F$
satisfies the sheaf condition for this cover. By (**), we may assume that $X$ is
Henselian local. Then $U_\bullet \to X$ is refined by a frét-cover $V \to X$ and
$F$ satisfies the sheaf condition with respect to $V \to X$ by assumption, so it
satisfies the sheaf condition with respect to $U_\bullet \to X$ by Lemma
\ref{lemm:sheaf-cond-technical}.
\end{proof}

\paragraph{Remark.} Using \cite[Corollary II.1.15]{andradas2012constructible}, the
claim (*) can be extended as follows: Every rét-cover $U_\bullet \to X$ with $X$
arbitrary is refined by a cover $V'_\bullet \to V_\bullet \to X$, where
$V_\bullet \to X$ is a Nisnevich cover and each $V'_i \to V_i$ is a frét-cover.

\begin{thm} \label{thm:rho-stable-htpy-mod}
Let $F_*$ be a homotopy module such that $\rho: F_n \to F_{n+1}$ is an
isomorphism for all $n$. Then $F_*$ consists of rét-sheaves.
\end{thm}
\begin{proof}
We apply Corollary \ref{corr:ret-criterion}. Hence let $\phi: U \to X$ be a
rét-cover with $\phi$ finite étale and $X$ essentially smooth, Henselian local.
We need to show that $F$ satisfies the sheaf condition with respect to this
cover.
Note that $U$ is then a finite disjoint union of essentially smooth, Henselian
local schemes, by \cite[Tag 04GH (1)]{stacks-project}.

We now use the transfer $tr: F_*(U) \to F_*(X)$ from Section
\ref{sec:recollections-motivic-homotopy}.
Any homotopy module is a module over $K_*^{MW}$ and satisfies the projection
formula with respect to this module structure. It follows from Corollary
\ref{corr:colim_KnMW} and our assumption that $\rho$ acts invertibly on $F_*$
that $F_*$ is a module over $a_{\ret}\ZZ$, and satisfies the projection formula
with respect to that module structure.

We know that for a Henselian local ring $A$ with residue field $\kappa$, we have
$H^0_{\ret}(A, \ZZ) = H^0_{\ret}(\kappa, \ZZ)$. This follows from \cite[Propositions II.2.2 and
II.2.4]{andradas2012constructible} (the author learned this argument from
\cite[proof of Lemma 6.4]{karoubi2015witt}). Consequently by
Corollary \ref{corr:transfer-aretZZ-surjective}, Proposition
\ref{prop:transfer-base-change} and stability of rét-covers under base change,
there exists $a \in H^0_{\ret}(U, \ZZ)$ such that $tr(a) = 1$.

Now suppose given $b \in F_*(X)$ such that $b|_U = 0$. Then $b = 1b =
tr(a) b = tr(a\cdot b|_U) = 0$ by the projection formula (i.e. Corollary
\ref{corr:transfer-projection-formula}). Consequently
$F_*(X) \to F_*(U)$ is injective.

Write $p_1, p_2: U \times_X U \to U$ for the two projections and suppose given
$b \in F_*(U)$ such that $p_1^*b = p_2^*b$. We have to show that there is $c \in
F_*(X)$ such that $b = c|_U$. I claim that $c := tr(ab)$ works. Indeed we have
$tr(ab)|_U = \phi^*(tr_\phi(ab)) = tr_{p_2}(p_1^*(ab))$ by Proposition \ref{prop:transfer-base-change}.
Now $p_1^*b =
p_2^*b$ by assumption, and so $tr_{p_2}(p_1^*(a) p_1^*(b)) = tr_{p_2}(p_1^*(a) p_2^*
(b)) = tr_{p_2}(p_1^*a)b$ by the projection formula again. Finally
$tr_{p_2}(p_1^*a) = \phi^* tr_\phi(a) = \phi^* 1 = 1$ by base change again, so
we are done.
\end{proof}

\section{Preliminary Observations}
\label{sec:preliminary}

We are now almost ready to prove our main theorems. This section collects some
preliminary observations and reductions.

Lemma \ref{lemm:rho-local-model} from Section \ref{sec:monoidal-local}
applies in particular to $\SH(S)$ and $D_\Aone(S)$ for a
Noetherian base scheme $S$. We will be particularly interested in the case $Y =
\Gm$ and $\alpha = \rho: S \to \Gm$ the additive inverse of the morphism
corresponding to -1. What the lemma says is that the $\rho$-localization can be
computed as the obvious colimit.

We write $\SH(S)^\ret$
for the real étale localisation of $\SH(S)$ and $D_\Aone(S)^\ret$ for the real
étale localisation of $D_\Aone(S)$. There is possibly a slight confusion as to
what this means, since it could mean the localisation at desuspensions of real
étale (hyper-) covers, or the category obtained by the same procedure as $\SH(S)$ but
replacing the Nisnevich topology by the real étale one from the start. This does
not actually make a difference:

\begin{lemm} \label{lemm:commute-loc-stab}
Let $\mathcal{M}$ be a monoidal model category, $T \in \mathcal{M}$ cofibrant
and $H$ a set of maps. There is an isomorphism of Quillen model categories
\[ Spt(L_H \mathcal{M}, T) = L_{H'} Spt(\mathcal{M}, T), \]
provided that all the localisations exist (e.g. $\mathcal{M}$ left proper and
combinatorial). Here $H' = \cup_{i \in \ZZ} \Sigma^{\infty + i} H$ and
$Spt(\mathcal{N}, U)$ denotes the model category of (non-symmetric) $U$-spectra
in $\mathcal{N}$ with the local model structure.
\end{lemm}
\begin{proof}
We follow \cite{hovey2001spectra}. Recall that $Spt(\mathcal{N}, U)$ denotes the
category of sequences $(X_1, X_2, X_3, \dots)$ together with bonding maps $X_i
\otimes U \to X_{i+1}$, and morphisms the compatible sequences of morphisms.
This is firstly
provided with a global model structure $Spt(\mathcal{N},U)_{gl}$
in which a map $(X_\bullet) \to
(Y_\bullet)$ is a fibration or weak equivalence if and only if $X_i \to Y_i$ is
for all $i$. This is also called a levelwise fibration or weak equivalence.
The local model structure is then obtained by localisation at a set
of maps which is not important to us, because it only depends on a choice of set of
generators of $\mathcal{M}$, and for $L_H \mathcal{M}$ we can just choose the
same generators.

Since in any model category $L_{H_1} L_{H_2} \mathcal{N} = L_{H_1 \cup H_2}
\mathcal{N}$, it is enough to show that $L_{H'} Spt(\mathcal{M}, T)_{gl} =
Spt(L_H \mathcal{M}, T)_{gl}$. Note that an acyclic fibration in $Spt(L_H
\mathcal{M}, T)_{gl}$ is the same as a levelwise acyclic $H$-local fibration in
$\mathcal{M}$, i.e. a levelwise acyclic fibration. Consequently the cofibrations
in $Spt(L_H \mathcal{M}, T)_{gl}$ are the same as in $Spt(\mathcal{M}, T)_{gl}$,
whereas the former has more weak equivalences. Thus the
former is a Bousfield localisation of the latter and hence it is enough to show that $L_{H'}
Spt(\mathcal{M}, T)_{gl}$ and $Spt(L_H \mathcal{M}, T)_{gl}$ have the same
fibrant objects. An object of $Spt(L_H \mathcal{M}, T)_{gl}$ is fibrant if and
only if it is levelwise $H$-locally fibrant. An object $E$ of $L_{H'}
Spt(\mathcal{M}, T)_{gl}$ is fibrant if and only if it is levelwise fibrant and
$H'$-local, which means that for each $\alpha: X \to Y \in H$ and every $n \in
\ZZ$ the map
\[ \Map^d(\Sigma^{\infty + n} \alpha, E): \Map^d(\Sigma^{\infty + n} Y, E)
      \to \Map^d(\Sigma^{\infty + n} Y, E) \]
is a weak equivalence. By adjunction, this is the same as $\Map^d(\alpha, E_n)$
being an equivalence, i.e. all $E_n$ being $H$-local. This concludes the proof.
\end{proof}

Write $\SH^{S^1}(S)$ for the $S^1$-stable homotopy category (i.e. obtained from motivic
spaces by just inverting $S^1$, but not $\Gm$).

\begin{lemm} \label{lemm:S1-rho-inversion}
There are canonical Quillen equivalences $\SH^{S^1}(S)[\rho^{-1}] \wequi
\SH(S)[\rho^{-1}]$ and similarly for the real étale topology.
\end{lemm}
\begin{proof}
By Lemma \ref{lemm:commute-loc-stab} we know that
$Spt(\mathcal{SH}^{S^1}(S)[\rho^{-1}], \Gm) = Spt(\mathcal{SH}^{S^1}(S),
\Gm)[\rho^{-1}]
\wequi \mathcal{SH}(S)[\rho^{-1}]$. But the map $\rho: S \to \Gm$ is invertible in
$\SH^{S^1}(S)[\rho^{-1}]$ and thus $Spt(\mathcal{SH}^{S^1}(S)[\rho^{-1}], \Gm) \wequi
\mathcal{SH}^{S^1}(S)[\rho^{-1}]$, i.e. inverting an invertible object has no
effect \cite[Theorem 5.1]{hovey2001spectra}
\end{proof}

We also observe the following:

\begin{prop} \label{prop:six-functors}
The pseudofunctor $X \mapsto
\SH(X)[\rho^{-1}]$ satisfies the full six functors formalism (on Noetherian
schemes of finite dimension), compact
generation, and continuity.
\end{prop}
\begin{proof}
If $i: Z \to X$ is a closed immersion then the
functor $i_*: \SH(Z) \to \SH(X)$ commutes with filtered homotopy colimits (being
right adjoint to a functor preserving compact objects) and
satisfies $i_*(X \otimes \Gm) \wequi i_*(X) \otimes \Gm$ \cite[A.5.1 (6) and
(3)]{triangulated-mixed-motives}.
It follows from the explicit description of $\rho$-localisation in Lemma
\ref{lemm:rho-local-model} that $i_*$ commutes with $L: \SH(X) \to
\SH(X)[\rho^{-1}]$. Thus $\SH(X)[\rho^{-1}]$ satisfies localisation, by
\cite[Proposition 2.3.19]{triangulated-mixed-motives}. Since
$\SH(X)[\rho^{-1}]$ clearly satisfies the homotopy and stability properties,
it satisfies the six functors formalism by Theorem
\ref{thm:fundamental-six-functors}.

Since $\SH(X)$ is compactly generated so is $\SH(X)[\rho^{-1}]$, by the last
sentence of Lemma \ref{lemm:rho-local-model}.

For any morphism $f: X \to Y$ the functor $f^*: \SH(Y) \to \SH(X)$ commutes with
(filtered) homotopy colimits (being a left adjoint), and consequently it
commutes with $\rho$-localisation, as above. Thus continuity for
$\SH(X)[\rho^{-1}]$ follows from continuity for $\SH(X)$.
\end{proof}

For completeness,  we include the following rather formal observation. It is not
used in the remainder of this text (except that it is restated as part of
Theorem \ref{thm:final-comparison}).

\begin{prop} \label{prop:ret-rho-alread-equiv}
The canonical functor $\SH(S)^\ret \to \SH(S)^\ret[\rho^{-1}]$
is an equivalence. In other
words, $\rho$ is a weak equivalence in $\SH(S)^\ret$.
\end{prop}
\begin{proof}
I claim that in $\SH(S)^\ret$ there is a splitting $\Gm \wequi \tunit \vee
\Delta$ such that the composite $\tunit \xrightarrow{\rho} \Gm \wequi \tunit
\vee \Delta \to \tunit$ is the identity.
It will follow from Lemma \ref{lemm:technical-monoidal} below that $\Delta \wequi 0$, proving
this lemma.

Call $a \in \mathcal{O}^\times(X)$ \emph{totally positive}
if for every real closed field $r$ and morphism $\alpha: Spec(r) \to X$ we have
$\alpha^*(a) > 0$. Note that in particular any square of a unit is totally positive.

This defines a sub-presheaf $G_+ \subset R_{\Aone \setminus
0}$ of the presheaf represented by $\Aone \setminus 0$.
Define $G_-$ analogously using totally negative units. I claim that $a_\ret
R_{\Aone \setminus 0} = a_\ret G_+ \coprod a_\ret G_{-1}$. We may prove this on
stalks, which are Henselian rings with real closed residue fields
\cite[(3.7.3)]{real-and-etale-cohomology}. If
$A$ is such a ring and $a \in A^\times$, then the reduction $\bar{a} \in A/m$ is
a unit and so either positive or negative. It follows that either $\bar{a}$ or
$-\bar{a}$ is a square, whence either $a$ or $-a$ is a square ($A$ being
Henselian of characteristic zero). Consequently $a$ is either totally positive or totally negative,
proving the claim.

We may thus define a map $a_\ret \Gm \to a_\ret S^0 = a_\ret(* \coprod *)$
by mapping $a_\ret G_+$ to the
base point and $a_\ret G_-$ to the other point. Since $-1$ is totally negative
this yields an unstable
splitting $a_\ret S^0 \to a_\ret \Gm \to a_\ret S^0$ of the required form. The stable
splitting follows.
\end{proof}

\begin{lemm}\label{lemm:technical-monoidal}
Let $\mathcal{C}$ be an additive symmetric monoidal category in which $\otimes$
distributes over $\oplus$.

If $G \in
\mathcal{C}$ is an invertible object, such that $G \iso \tunit \oplus \Delta$,
then $\Delta \iso 0$.
\end{lemm}
\begin{proof}
The object $G$ is rigid (being invertible) and hence
$\Delta$ is rigid (being a summand of $G$). We have $\tunit \iso D(G) \otimes G \iso (D\tunit \oplus
D\Delta) \otimes (\tunit \oplus \Delta) \iso \tunit \oplus \Delta \oplus
D(\Delta) \oplus D(\Delta) \otimes \Delta$. Thus in order to prove the claim we
may assume that $G = \tunit$. Now the splitting $\tunit \iso \tunit \oplus
\Delta$ corresponds to morphisms $\tunit \xrightarrow{e} \tunit \oplus \Delta
\xrightarrow{f} \tunit$ with $fe = \id$ and $ef \in End(\tunit \oplus \Delta)$
the projection. Fixing an isomorphism $\tunit
\xrightarrow{z} \tunit \oplus \Delta$ we get corresponding elements $z^{-1}e, fz
\in [\tunit,\tunit]$. We have $\id = fe = f(zz^{-1})e = (fz)(z^{-1}e)$. But
$End(\tunit)$ is commutative \cite[sentence before Proposition 2.2]{balmer2010spectra} so $\id = (z^{-1}e)(fz)$ and
consequently $ef = zz^{-1} = \id$ and $\Delta = 0$.
\end{proof}

\section{Main Theorems}
\label{sec:main-theorems}

\begin{prop} \label{prop:main-case-char-0}
Let $k$ be a field of characteristic zero. The functor $L: \SH(k)[\rho^{-1}] \to
\SH(k)^\ret[\rho^{-1}]$ is an equivalence.
\end{prop}
\begin{proof}
It is
enough to show that all objects in $\SH(k)[\rho^{-1}]$ are rét-local. Let
$U_\bullet \to X$ be a rét-hypercover, and let $\hat{X}$ be its homotopy colimit
(in $\SH(k)$). We need to show that if $E \in \SH(k)[\rho^{-1}]$, then
$[\hat{X}, E] = [X, E]$. We have conditionally convergent descent
spectral sequences
\begin{equation}\label{eq:ss1}
 H^p_{Nis}(X, \ul{\pi}_{-q}(E)_{-i}) \Rightarrow [\Sigma^\infty X_+ \wedge
\Gmp{i}, E[p+q]]
\end{equation}
\begin{equation}\label{eq:ss2}
 [\hat{X}, \ul{\pi}_{-q}(E)_{-i}[p]] \Rightarrow [\hat{X} \wedge
\Gmp{i}, E[p+q]].
\end{equation}
Here we display the $E_2$-pages on the left hand side.
We moreover have the conditionally convergent homotopy colimit spectral sequence
\begin{equation}\label{eq:ss3}
  [U_*, \ul{\pi}_{-q}(E)_{-i}[p]] \Rightarrow [\hat{X},
\ul{\pi}_{-q}(E)_{-i}[p+*]].
\end{equation}
Here the left hand side is the $E_1$-page.
We have $[U_n, \ul{\pi}_{-q}(E)_{-i}[p]] = H^p_{Nis}(U_n, \ul{\pi}_{-q}(E)_{-i})
= H^p_\ret(U_n, \ul{\pi}_{-q}(E)_{-i})$; indeed since $E$ is $\rho$-local each
$\ul{\pi}_{-q}(E)_{-i}$ is a rét-sheaf, by Theorem
\ref{thm:rho-stable-htpy-mod}, and for any rét-sheaf $F$ we
have $H^p_{ret}(U_n, F) = H^p_{Nis}(U_n, F)$
\cite[Proposition 19.2.1]{real-and-etale-cohomology}. It follows that spectral
sequence \eqref{eq:ss3} converges strongly (because the dimension of $X$ is
finite) and identifies with the descent
spectral sequence in rét-cohomology for the cover $U_\bullet$. In particular, it
converges to $H^{p+*}_\ret(X, \ul{\pi}_{-q}(E)_{-i})$. Thus  we find that
$[\hat{X}, \ul{\pi}_{-q}(E)_{-i}[p]] = H^p_\ret(X, \ul{\pi}_{-q}(E)_{-i})$. Using
\cite[Proposition 19.2.1]{real-and-etale-cohomology} again, we conclude that the
evident map from spectral sequence \eqref{eq:ss1} to spectral sequence
\eqref{eq:ss2} induces an isomorphism on the $E_2$-pages, and moreover both
converge strongly (again for cohomological dimension reasons). Thus the induced
map on targets is an isomorphism, which is what we wanted to show.
\end{proof}

\begin{corr} \label{corr:main-case-fields}
The proposition holds for all fields.
\end{corr}
\begin{proof}
I claim that if $k$ has positive characteristic, then $\rho$ is nilpotent in
$\SH(k)$. By base change, it suffices to prove this when $k = \FF_p$.
That is to say, we wish to show that $\rho$
is nilpotent in $K_*^{MW}(\FF_p)$, or equivalently that
$\colim_n K_n^{MW}(\FF_p) = 0$. By the same argument as in the proof of
Corollary \ref{corr:colim_KnMW} we know that $\colim_n K_n^{MW}(\FF_p)] =
\colim_n I^n(\FF_p)$. Thus our claim follows from
nilpotence of the fundamental ideal of $\FF_p$, which is well known \cite[III (5.9)]{milnor1973symmetric}.
\end{proof}

\begin{corr} \label{corr:main-equiv}
Let $S$ be a Noetherian scheme of finite dimension. The functor $L:
\SH(S)[\rho^{-1}] \to \SH(S)^\ret[\rho^{-1}]$ is an equivalence.

In particular $\SH(S)^\ret[\rho^{-1}]$ satisfies the full six functors
formalism.
\end{corr}
Our initial proof of this statement contained a mistake; a correction and vast
simplification has kindly been communicated by Denis-Charles Cisinski.
\begin{proof}
It suffices to prove that all objects of $\SH(S)[\rho^{-1}]$ are rét-local. Thus
let $X \in Sm(S)$ and $U_\bullet \to X$ a rét-hypercover. We need to show that
\[ \alpha: \hocolim_\Delta \Sigma^\infty U_\bullet \to \Sigma^\infty X\]
is an
equivalence in $\SH(S)[\rho^{-1}]$. (See also Lemma \ref{lemm:commute-loc-stab}.)
Since $\SH(S)[\rho^{-1}]$ satisfies the six
functors formalism by Proposition \ref{prop:six-functors}, it follows from Corollary
\ref{corr:conservative-detection} that suffices to show that if $f: Spec(k) \to S$
is a morphism (with $k$ a field), then $f^*\alpha$ is an equivalence. But
$f^*$ is a left adjoint so commutes with homotopy colimits (and $\Sigma^\infty$), so
$f^*\alpha$ is isomorphic to the map
\[ \hocolim_\Delta \Sigma^\infty f^*U_\bullet \to \Sigma^\infty f^* X \]
in $\SH(k)[\rho^{-1}]$. Since rét-covers are stable by pullback,
this is an equivalence by Corollary \ref{corr:main-case-fields}.
\end{proof}

\begin{prop} \label{prop:main-equiv}
Let $S$ be a Noetherian scheme of finite dimension. Then the canonical functor
$\sSH(S_\ret) \to \SH(S)^\ret[\rho^{-1}]$ is an equivalence.
\end{prop}
\begin{proof}
The functor $\sSH(S_\ret) \to \sSH(Sm(S)_\ret)$ is fully faithful and $t$-exact by
Corollary \ref{corr:Le-exact-ff}.
The image of $\sSH(S_\ret)$ in $\sSH(Sm(S)_\ret)$ consists of $\Aone$-local and
$\rho$-local objects, by the descent spectral sequence and
Theorem \ref{thm:ret-htpy-rho-inv} (and Corollary \ref{corr:Le-exact-ff}, which
implies that the
homotopy sheaves of $LeE$ are the extensions of the homotopy sheaves of $E$).
Consequently
$\sSH(S_\ret) \to \SH^{S^1}(S)^\ret[\rho^{-1}]$ is fully faithful. But
$\SH^{S^1}(S)^\ret[\rho^{-1}] \to \SH(S)^\ret[\rho^{-1}]$ is an
equivalence by Lemma \ref{lemm:S1-rho-inversion}. We have thus established
that the functor is fully faithful. We need to show it is essentially
surjective.

The category $\SH(S)^\ret[\rho^{-1}]$ is generated by objects of the form
$p_*(\tunit)$ where $p: T \to S$ is projective \cite[Proposition
4.2.13]{triangulated-mixed-motives}. Since the functor $e:
\sSH(S_\ret) \to \SH(S)^\ret[\rho^{-1}]$ has a right adjoint it commutes with
arbitrary sums, and hence it identifies $\sSH(S_\ret)$ with a localising
subcategory of $\SH(S)^\ret[\rho^{-1}]$. It thus suffices to show that $e$
commutes with $p_*$, where $p: T \to S$ is a projective morphism. This is
exactly the same as the proof of \cite[Proposition 4.4.3]{cisinski2013etale}.
It boils down to the
proper base change theorem holding both in $\SH(S)^\ret[\rho^{-1}]$ (where it
follows from the six functors formalism which we have already established by
showing that $\SH(S)^\ret[\rho^{-1}] \wequi \SH(S)[\rho^{-1}]$) and in
$\sSH(S_\ret)$; the latter is Theorem \ref{thm:ret-proper-basechange}.
\end{proof}

\paragraph{Remark.} If $S$ is the spectrum of a field, the above proof can be
simplified greatly, by arguing as in \cite[Section 5]{bachmann-hurewicz}. See in
particular Lemma 21, Corollary 26 and Proposition 28 of \emph{loc. cit.} This way we no
longer need to use the proper base change theorems, and thus also do not need to
know that $\SH(X)^\ret[\rho^{-1}]$ satisfies the six functors formalism.

One may also extract from \emph{loc. cit.} a proof of
Proposition \ref{prop:main-case-char-0} not relying on Theorem
\ref{thm:rho-stable-htpy-mod}. Thus if the base is a field, Sections
\ref{sec:recoll-real-etale}, \ref{sec:recoll-pre-motivic}, and
\ref{sec:jacobson} can be dispensed with.

\paragraph{}In summary, we have thus established the following result.

\begin{thm} \label{thm:final-comparison}
Let $S$ be a Noetherian scheme of finite dimension.
In the following two diagrams, all functors are the canonical ones, and are
equivalences of categories:
\begin{equation*}
\begin{CD}
\SH(S_\ret) @>a>> \SH^{S^1}(S)^\ret[\rho^{-1}] @<<< \SH^{S^1}(S)[\rho^{-1}]\\
@VVV                   @VbVV                          @Vb'VV  \\
\SH(S)^\ret       @>c>> \SH(S)^\ret[\rho^{-1}] @<d<< \SH(S)[\rho^{-1}]
\end{CD}
\end{equation*}

\begin{equation*}
\begin{CD}
D_\Aone(S_\ret) @>>> D_\Aone^{S^1}(S)^\ret[\rho^{-1}] @<<< D_\Aone^{S^1}(S)[\rho^{-1}] \\
@VVV                   @VVV                                  @VVV  \\
D_\Aone(S)^\ret       @>>> D_\Aone(S)^\ret[\rho^{-1}] @<<< D_\Aone(S)[\rho^{-1}].
\end{CD}
\end{equation*}

In particular all these categories satisfy the full six functors formalism, and
continuity.
\end{thm}
\begin{proof}
The functor $d$ is an equivalence by Corollary \ref{corr:main-equiv}, $b$ and
$b'$ are equivalences by Corollary \ref{lemm:S1-rho-inversion}, $c$ is an
equivalence by Proposition \ref{prop:ret-rho-alread-equiv} and $ba$ is an
equivalence by Proposition  \ref{prop:main-equiv}. It follows that $a$ is an
equivalence, and so are the two unlabelled functors.

By Proposition \ref{prop:six-functors}, $\SH(\bullet)[\rho^{-1}]$ satisfies the
full six functors formalism, and hence so do all the other pseudofunctors, being
equivalent.

We have provided the proofs for $\SH$, the ones for $D_\Aone$ are exactly the
same.
\end{proof}

\section{Real Realisation}
\label{sec:real-realisation}

In this section we work over the field $\RR$ of real numbers. We then have a
composite
\[ R_1: \SH(\RR) \xrightarrow{L_\rho} \SH(\RR)^\ret[\rho^{-1}] \wequi
   \SH^{S^1}(\RR)^\ret[\rho^{-1}] \xrightarrow{r} \SH^s. \]
Here by $\SH^s$ we mean the model of the stable homotopy category $\SH$ built
from simplicial sets. Of course $\SH^s \wequi \SH$ canonically (and this may be
an equality depending, on our favourite model of $\SH$). Also $r$ denotes the
functor induced by the right adjoint of $e: Pre(\RR_\ret) \to Pre(Sm(\RR))$ from
Section \ref{sec:recoll-real-etale}.

Following Heller-Ormsby \cite[Section 4.4]{heller2016galois}, there is also the
real realisation functor $LR_2: \SH(\RR) \to \SH^t$. Here $\SH^t$ is the model of
$\SH$ built from topological spaces. The functor $LR_2$ is defined by starting
with the functor $R_2: Sm(\RR) \to Top, X \mapsto X(\RR)$ assigning a smooth
scheme over $\RR$ its set of real points with the strong topology. We then get a
functor $R_2: sPre(Sm(\RR))_* \to Top$ by left Kan extension, i.e. demanding that
$R_2(\Delta^n_+ \wedge X_+) = \Delta^n_+ \wedge X(\RR)_+$ and that $R_2$ preserves
colimits. Using the projective model structure on $sPre(Sm(\RR))_*$ this functor
is left Quillen and then one promotes it to $LR_2: \SH(\RR) \to \SH^t$ in the
usual way.

Fortunately the two potential real realisation functors are the same. To state
this result, recall that there is an adjunction
\[ |\bullet| : sSet \leftrightarrows Top : Sing^t, \]
and then by passing to homotopy categories of spectra one obtains the adjoint equivalence
\[ L|\bullet| : \SH^s \leftrightarrows \SH^t : RSing^t. \]

\begin{prop} \label{prop:real-realn-compatible}
The two functors $L|R_1(\bullet)|, LR_2(\bullet): \SH(\RR) \to \SH^t$ are
canonically isomorphic.
\end{prop}
\begin{proof}
The functor $R_2$ takes multiplication by $\rho$ into a weak equivalence.
Consequently it remains left Quillen in the $\rho$-local model structure and
hence $LR_2$ canonically factors through the localisation $\SH(\RR) \to
\SH(\RR)[\rho^{-1}].$ Since $\SH(\RR)[\rho^{-1}] \wequi
\SH^{S^1}(\RR)^\ret[\rho^{-1}]$ the obvious functor $R_2': Spt(Sm(\RR)) \to
Spt^t$ is left Quillen in the $(\rho,\ret,\Aone)$-local model structure. (Here
we have used twice the following well-known observation: if $L: \mathcal{M}
\leftrightarrows \mathcal{N}: R$ is a Quillen adjunction and $H$ is a set of
maps between cofibrant objects in $\mathcal{M}$ which is taken by $L$ into weak
equivalences, then $L: L_H \mathcal{M} \leftrightarrows \mathcal{N}: R$ is also
a Quillen adjunction. This follows from \cite[Propositions 8.5.4 and
3.3.16]{hirschhorn-book}.)

We now have the following diagram (which we do not know to be commutative so far):
\begin{equation*}
\begin{CD}
\SH^{S^1}(\RR)^\ret[\rho^{-1}] @>R_2'>> \SH^t \\
@VrVV                                     @|    \\
\SH^s     @>|\bullet|>>                   \SH^t.
\end{CD}
\end{equation*}

Here all the functors are derived; we omit the ``L'' and ``R''.
The functor $r$ is an equivalence with inverse $e$ by Theorem
\ref{thm:final-comparison}.
Thus for $E \in \SH^{S^1}(\RR)^\ret[\rho^{-1}]$ we have a canonical
isomorphism $R_2' E \wequi R_2' erE$ and so to prove the proposition it suffices to
exhibit a canonical isomorphism of functors $R_2' e \wequi |\bullet|$. 

But this isomorphism exists on the level of underived functors, and then passes
to the homotopy categories. Indeed if $E \in Spt^s$ then $R_2'eE$ and $|E|$ are
both computed by applying functors (of the same names) levelwise to $E$, so we
may just as well show that for $E \in sSet_*$ we have $R_2'eE \iso |E|$. But now
$R_2'$, $e$ and $|\bullet|$ all preserve colimits, so we may just deal
with $E = \Delta^n_+$. But then $R_2'eE \Delta^n_+ = |\Delta^n_+|$ holds
basically by definition.
\end{proof}

A similar result can be obtained for the $\Aone$-derived category. We have $r:
D_\Aone(\RR)[\rho^{-1}] \to D(Spec(\RR)_\ret) \wequi D(Ab).$ There is also
$R_2: D_\Aone(\RR) \to D(Ab)$ which is obtained by (derived) left Kan extension
from the functor $Sm(\RR) \to C(Ab)$ which sends a smooth scheme $X$ the singular complex of its real points
$C_*(X(\RR))$. Then there is a commutative diagram
\begin{equation*}
\begin{CD}
\SH(\RR) @>R_2>> \SH     \\
@VC_*VV          @VC_*VV \\
D_\Aone(\RR) @>R_2>> D(Ab).
\end{CD}
\end{equation*}

\begin{prop}
The functors $rL_\rho, R_2: D_\Aone(\RR) \to D(Ab)$ are canonically isomorphic.
\end{prop}
\begin{proof}
As before $R_2$ factors through $L_\rho$ as $R_2'$ and we may show that $r,
R_2': D_\Aone(\RR)[\rho^{-1}] \to D(Ab)$ are canonically isomorphic.
The functor $r$ is an equivalence with inverse $e$, so it is enough to show that
$R_2' e: D(Ab) \to D_\Aone(\RR)[\rho^{-1}] \wequi
D_\Aone^{S^1}(\RR)[\rho^{-1}] \to D(Ab)$ is canonically isomorphic to the
identity. This is the same argument as before.
\end{proof}

Let us make explicit the following consequence.

\begin{corr}
Let $E \in \SH(\RR)$. Then
\[ \ul{\pi}_i(E)(\RR)[\rho^{-1}] = \pi_i(RE) \]
and
\[ \ul{h}_i^\Aone(E)(\RR)[\rho^{-1}] = H_i(RE). \]
Here $R: \SH(\RR) \to \SH$ denotes any one of the (canonically isomorphic) real
realisation functors we have considered and $\ul{h}_i^\Aone(E) := \ul{h}_i(FE)$
where $F: \SH(\RR) \to D_\Aone(\RR)$ is the canonical functor.
\end{corr}
\begin{proof}
Combine Lemma \ref{lemm:rho-local-model} (saying that $L_\rho E =
E[\rho^{-1}] = \hocolim_n E \wedge \Gmp{n}$) with compactness of the
units of $\SH, D_\Aone$ and the above two propositions.
\end{proof}

\section{Application 1: The $\eta$-inverted Sphere}
\label{sec:app1}

From now on, $k$ will denote a perfect field. Since essentially all our results
concern the $\rho$-inverted situation, they are really only interesting if $k$
has characteristic zero, so this is not a big restriction.

Recall that the motivic Hopf map $\eta: \mathbb{A}^2 \setminus 0 \to
\mathbb{P}^1$ defines an element of the same name in motivic stable homotopy
theory $\eta: \Sigma^\infty \Gm \to \tunit$. Here we use that $\Sigma^\infty
(\mathbb{A}^2 \setminus 0) \wequi \Sigma^\infty \Gm \wedge \Sigma^\infty
\mathbb{P}^1$. The element $\eta \in \ul{\pi}_0(\tunit)_{-1}$ is non-nilpotent, and
so inverting it is very natural. The category $\SH(k)[\eta^{-1}]$ can be
constructed very similarly to $\SH(k)[\rho^{-1}]$. In particular the
localisation functor $L: \SH(k) \to \SH(k)[\eta^{-1}]$ is just the evident
colimit, see Lemma \ref{lemm:rho-local-model}. It is typically denoted $E
\mapsto E_\eta$ or $E \mapsto E[1/\eta]$.
One may similarly invert other endomorphisms of the sphere
spectrum. If $0 \ne n \in \ZZ$ then there is a corresponding automorphism of
$\tunit$, and we denote the localisation by $E \mapsto E[1/n]$.

At least after inverting $2$, inverting $\eta$ is essentially
the same as inverting $\rho$:

\begin{lemm} \label{lemm:rho-eta-comparison}
The endomorphism ring $K_*^{MW}(k)[1/2] = [\tunit[1/2], \tunit[1/2] \wedge \Gmp{*}]$ splits canonically
into two summands $K_*^{MW}(k)[1/2] = K^+ \oplus K^-$. In fact $K^- =
K_*^{MW}(k)[1/2, 1/\eta]$ and $K^+$ is characterised by the fact that $\eta K^+
= 0$.

In $K^-$ we have the equality $\eta \rho = 2$, whereas in
$K^+$ we have $\rho^2 = 0$. In particular
\[ K_*^{MW}(k)[1/2, 1/\eta] = K^- = K_*^{MW}(k)[1/2, 1/\rho]. \]
\end{lemm}
\begin{proof}
This is well known, see for example \cite[Section 3.1]{A1-alg-top}. We
summarise: For $a \in k^\times$ let $\langle a \rangle = 1 + \eta [a] \in
K_0^{MW}(k)$. Put $\epsilon = -\langle-1\rangle$. Then $\epsilon^2 = 1$ and so
after inverting $2$, $K_*^{MW}(k)$ splits into the eigenspaces for $\epsilon$.
One puts $h = 1 - \epsilon$ and then has $\eta h = 0$. On $K^+$ we have
$\epsilon = -1$, so $h = 2$ and consequently $\eta = 0$ (since $2$ is
invertible).

By definition $\rho = -[-1]$ and consequently $\eta \rho = 1 + \epsilon$.
Thus on $K^-$ where $\epsilon = 1$ we find $\eta \rho = 2$ as claimed, and in
particular $\eta$ is invertible on $K^-$.

Finally $\rho^2 h = 0$ in $K_*^{MW}(k)$ and thus $2 \rho^2 = 0$ in $K^+$.
(This is just another expression of the fact that $K^+ \iso K^M(k)[1/2]$ is
graded-commutative and $\rho$ has degree 1, so $\rho^2 = -\rho^2$.) But since $2$
is invertible in $K^+$ we find $\rho^2 = 0$ (in $K^+$). This concludes the proof.
\end{proof}

Oliver Röndigs has studied the homotopy sheaves $\ul{\pi}_1(\tunit_\eta)$ and
$\ul{\pi_2}(\tunit_\eta)$ and proved that they vanish \cite{rondigs2016eta}. (Note
that $\ul{\pi}_i(E_\eta)_*$ is independent of $*$, because multiplication by
$\eta$ is an isomorphism, so we shall suppress the second index.) He argues that
$\ul{\pi}_i(\tunit)_* \to \ul{\pi}_i(\tunit[1/2])_*$ is injective for $i = 1, 2$ (see his
Lemma 8.2) and consequently an important part of his work is in showing that
$\ul{\pi}_i(\tunit[1/\eta, 1/2]) = 0$ for $i = 1, 2$. We can deduce this as an easy
corollary from our work:

\begin{prop} \label{prop:app-rondigs}
Let $k$ be a perfect field. Then $\ul{\pi}_i(\tunit[1/\eta, 1/2]) = 0$ for $i = 1, 2$.
\end{prop}
\begin{proof}
By Lemma \ref{lemm:rho-eta-comparison} we know that $\SH(k)[1/2, 1/\eta] =
\SH(k)[1/2, 1/\rho]$. By Theorem \ref{thm:final-comparison}, we have
$\SH(k)[1/2, 1/\rho] = \sSH(Spec(k)_\ret)[1/2]$. In particular this category is
trivial unless $k$ has characteristic zero, which we shall assume from now on.

The sheaves $\ul{\pi}_i(\tunit[1/2, 1/\rho])$ are unramified \cite[Lemma 6.4.4]{morel2005stable}, so it
suffices to show that $\ul{\pi}_i(\tunit[1/2, 1/\rho])(K) = 0$ for $i=1,2$ and every $K$ (of
characteristic zero). Since $k$ was also arbitrary, we may just as well show the
result for $k=K$, simplifying notation. We are dealing with rét-sheaves by
Theorem
\ref{thm:rho-stable-htpy-mod}, and so if
\[ Spec(l_1) \coprod Spec(l_2) \coprod \dots \coprod Spec(l_n) \to Spec(k) \]
is a rét-cover, the canonical map
\[ \ul{\pi}_i(\tunit[1/2, 1/\rho])(k) \to \prod_m \ul{\pi}_i(\tunit[1/2, 1/\rho])(l_m) \]
is injective. Consequently we may assume that $k$ is real closed. But then
$\sSH(Spec(k)_\ret) = \SH$ is just the ordinary stable homotopy category, so it
suffices to show: $\pi_i^s[1/2] = 0$ for $i=1, 2$, where $\pi_i^s$ are the
classical stable homotopy groups. But $\pi_1^s = \ZZ/2 = \pi_2^s$ is well known,
so we are done.
\end{proof}

In classical algebraic topology, it is well known that rational stable homotopy
theory is the same as rational homology theory: $\SH_\QQ \wequi D(\QQ)$. In
motivic stable homotopy theory, the situation is not so simple. As is well known
(and follows for example from Lemma \ref{lemm:rho-eta-comparison}) there is a
splitting $\SH(k)_\QQ = \SH(k)_\QQ^+ \times \SH(k)_\QQ^-$. The + part has been
identified with rational motivic homology theory by Cisinski-Déglise
\cite[Section 16]{triangulated-mixed-motives}: $\SH(k)_\QQ^+ \wequi \DM(k, \QQ)$.

The - part has been only identified recently with an appropriate category of
rational Witt motives by Ananyevskiy-Levine-Panin \cite{levine2015witt}:
$\SH(k)_\QQ^- = \DM_W(k, \QQ)$. Here the category $\DM_W(k, \QQ)$ may be
conveniently defined as the homotopy category of modules over the (strict ring
spectrum model of the) homotopy module of rational Witt theory. That is to say
there is the homotopy module $\ul{W}_\QQ$ with $(\ul{W}_\QQ)_i = \ul{W}
\otimes_\ZZ \QQ$ for all $i$. This is the same as $K^- \otimes_{\ZZ[1/2]} \QQ$,
or equivalently $\ul{\pi}_0(\tunit_{\eta,\QQ})$.
Then corresponding to this homotopy module there is a strict ring spectrum, the
Eilenberg-MacLane spectrum $EM \ul{W}_\QQ$. Finally we may form the model
category $EM\ul{W}_\QQ\Mod$ and its homotopy category $Ho(EM\ul{W}_\QQ\Mod) =:
\DM_W(k, \QQ)$.

More generally, one may define $\DM_W(k, \ZZ[1/2])$ by replacing $\ul{W}_\QQ =
\ul{\pi}_0(\tunit_\QQ[1/\eta])$ in the above construction with $\ul{W}[1/2] =
\ul{\pi}_0(\tunit[1/\eta,1/2])$.

The theorem of Ananyevskiy-Levine-Panin essentially boils down to the
computation that $\ul{\pi}_i(\tunit_{\eta,\QQ}) = 0$ for $i > 0$. We can deduce
this and more from our general theory. We write $H_{\Aone} \ZZ$ for the image of
the tensor unit in $D_\Aone(k)$ under the ``forgetful'' functor $D_\Aone(k) \to
\SH(k)$.

\begin{prop} \label{prop:app1-ident}
We have $\ul{\pi}_i(H_{\Aone} \ZZ[1/\rho]) = 0$ for $i > 0$, and consequently
$\ul{\pi}_i(H_{\Aone} \ZZ[1/2, 1/\eta]) = 0$ for $i > 0$. Similarly
$\ul{\pi}_i(\tunit_\QQ[1/\rho]) = \ul{\pi}_i(\tunit_\QQ[1/\eta]) = 0$ for $i > 0$.

Thus we have the equivalences
\[ D_\Aone(k, \ZZ[1/2])^- \wequi \DM_W(k, \ZZ[1/2]) \wequi D(Spec(k)_\ret, \ZZ[1/2]) \]
\[ \SH(k)_\QQ^- \wequi \DM_W(k, \QQ) \wequi D(Spec(k)_\ret, \QQ). \]
\end{prop}
\begin{proof}
As in the proof of Proposition \ref{prop:app-rondigs} we have
\[ D_\Aone(k)[1/2]^- := D_\Aone(k)[1/2, 1/\eta] = D_\Aone(k)[1/2, 1/\rho]. \]
By Theorem \ref{thm:final-comparison}, this is the same as $D(Spec(k)_\ret)[1/2]
= D(Spec(k)_\ret, \ZZ[1/2])$.

Similarly
\[ \SH(k)_\QQ^- := \SH(k)_\QQ[1/\eta] = \SH(k)_\QQ[1/\rho] =
   \sSH(Spec(k)_\ret)_\QQ, \]
and the latter category is equivalent to $D(Spec(k)_\ret)_\QQ$ by classical
stable rational homotopy theory.

From this we can read off $\ul{\pi}_*(H_\Aone \ZZ[1/2, 1/\eta])$ and so on. The
main point is that $\ul{\pi}_n(H_\Aone \ZZ[1/2, 1/\eta]) = 0$ for $n > 0$. It
suffices to check this on fields, so we may as well check it for $k$ ($k$ being
arbitrary), and we have $\ul{\pi}_n(H_\Aone \ZZ[1/2, 1/\eta])(k) = [\tunit[n],
\tunit]_{D(Spec(k)_\ret, \ZZ[1/2])} = H^{-n}_\ret(k, \ZZ[1/2]) = 0$.

It remains
to show that $D_\Aone(k, \ZZ[1/2])^- \simeq \DM_W(k, \ZZ[1/2])$. Our
computation of homotopy sheaves implies that $H_\Aone\ZZ[1/\rho, 1/2] \to EM\ul{W}[1/2]$
is a weak equivalence. The result follows.
\end{proof}

Let us also make explicit the following observation.

\begin{corr}
Let $k$ be a real closed field or $\QQ$. Then $\ul{\pi}_*(\tunit[1/\rho])(k) =
\pi_*^s$ and in particular $\ul{\pi}_*(\tunit[1/\eta, 1/2])(k) = \pi_*^s
\otimes_\ZZ \ZZ[1/2]$. Here $\pi_*^s$ denotes the classical stable homotopy
groups.
\end{corr}
\begin{proof}
This follows immediately from $Shv(Spec(k)_\ret) = Set$, Lemma
\ref{lemm:rho-eta-comparison} and Theorem \ref{thm:final-comparison}.
\end{proof}

\section{Application 2: Some Rigidity Results}
\label{sec:app2}

In this section we establish some rigidity results. We
work with $\rho$-stable sheaves. These sheaves are $h$-torsion (because $\rho^2
h = 0$), explaining
to some extent why we do not need the usual torsion assumptions.

There are various notions of rigidity for sheaves. We shall call a presheaf $F$
on $Sm(k)$ \emph{rigid} if for every essentially smooth, Henselian local scheme
$X$ with residue field $x$, the natural map $F(X) \to F(x)$ is an isomorphism.
This notion goes back to perhaps Gillet-Thomason \cite{gillet1984k} and Gabber \cite{gabber1992k}.

\begin{lemm} \label{lemm:loc-const-rigid-ret}
Let $F \in Shv(Spec(k)_\ret)$. Then $eF \in Shv(Sm(k)_\ret)$ is rigid.
\end{lemm}
\begin{proof}
Extension $e$ and pullback are both left Kan extensions. From this it is easy to
show that they commute,
and so we find that $(eF)|_{X_\ret} = f^*F \in Shv(X_\ret)$, where $X$ is
(essentially) smooth over $k$ with structural morphism $f$.

If $char(k) > 0$ then $Spec(k)_\ret$ and $Sm(k)_\ret$ are both the trivial site,
so we may assume that $k$ is of characteristic zero and consequently perfect. In
this case, for an essentially $k$-smooth Henselian local scheme $X$ with closed
point $i\colon x \to X$, there exists a retraction $s\colon X \to x$. (Write
$k(x)/k$ as $k(T_1, \dots, T_n)[U]/P$ with $P \in k(T_1, \dots, T_n)[U]$ separable;
this is possible because $k(x)/k$ is separable, $k$ being perfect. Lift the
elements $T_i$ to $\mathcal{O}_X$ arbitrarily and then use Hensel's lemma to
produce a root of $P$ in $\mathcal{O}_X$.)

It is thus enough to prove: if $F \in Shv(x_\ret)$ then $H^0(x, F) = H^0(X,
s^*F)$. It follows from \cite[Discussion after Proposition
19.2.1]{real-and-etale-cohomology} and \cite[Propositions II.2.2 and
II.2.4]{andradas2012constructible} that for any $G \in Shv(X_\ret)$ we have
$H^0(X, G) = H^0(x, i^*G)$. Consequently $H^0(X, s^*F) = H^0(x, i^*s^*F) = H^0(x,
F)$, because $si = \id$ by construction.
\end{proof}

\begin{corr} \label{corr:rho-rigidity}
If $E \in \SH(k)[\rho^{-1}]$ then all the homotopy sheaves $\ul{\pi}_i(E)$ are
rigid.
\end{corr}
\begin{proof}
By Theorem \ref{thm:final-comparison} and Corollary \ref{corr:Le-exact-ff} we
know that all the homotopy sheaves of $E$ are of the form $eF$, with $F \in
Shv(Spec(k)_\ret)$. Thus the claim follows immediately from Lemma
\ref{lemm:loc-const-rigid-ret}.
\end{proof}

\begin{corr}
Let $k$ be a perfect field of finite virtual 2-étale cohomological dimension and
exponential characteristic $e \ne 2$.
Then the homotopy sheaves $\ul{\pi}_i(\tunit)_0[1/e]$ are rigid.
\end{corr}
\begin{proof}
We will first assume that $e=1$, and explain the necessary changes in positive
characteristic at the end.

For $i=0$ we have $\ul{\pi}_0(\tunit)_0 = \ul{GW}$ and this sheaf is known to be
rigid \cite[Theorem 2.4]{gille2015quadratic}. We consider the arithmetic square \cite[Lemma
3.9]{rondigs2016first}
\begin{equation*}
\begin{CD}
\tunit @>>> \tunit[1/2] \\
@VVV    @VVV  \\
\tunit_2^\wedge @>>> \tunit_2^\wedge[1/2].
\end{CD}
\end{equation*}
Since rigid sheaves are stable under extension, kernel and cokernel, the five
lemma implies that it
is enough to show that $\ul \pi_*(\tunit[1/2])_0, \ul \pi_*(\tunit_2^\wedge)_0$ and $\ul
\pi_*(\tunit_2^\wedge[1/2])_0$ are rigid. Since rigid sheaves are stable by colimit,
the case of $\tunit_2^\wedge[1/2]$ follows from $\tunit_2^\wedge$.

By \cite[Theorem 1]{hu2011convergence} and \cite[proof of
Theorem 8.1]{rondigs2016eta}, we know that $\tunit_2^\wedge$ is the target of the
convergent motivic Adams spectra sequence. The homotopy sheaves at the
$E_1$ page are all sheaves with transfers in the sense of Voevodsky and
torsion prime to the characteristic, and hence
rigid, for example by \cite[Paragraph after Lemma 1.6]{hornbostel2007rigidity}.
Since rigid sheaves are stable by extension etc., it follows that the $E_\infty$
page is rigid, and finally so are the homotopy sheaves of $\tunit_2^\wedge$.

By motivic Serre finiteness \cite[Theorem
6]{levine2015witt} (beware that their indexing convention for motivic homotopy
groups differs from ours!), $\ul \pi_i(\tunit[1/2])_0$ is torsion for $i>0$.
By design, it is of odd
torsion prime to the exponential characteristic. Consequently all of the
$l$-torsion subsheaves of $\ul{\pi}_i(\tunit[1/2]^+)_0$ are rigid by the same argument
as before, and so is the colimit $\ul{\pi}_i(\tunit[1/2]^+)_0$.

It remains to deal with $\ul{\pi}_i(\tunit[1/2]^-)_0$. But this is just the same as
$\ul{\pi}_i(\tunit[1/2, 1/\rho])_0$ and so is rigid by Corollary
\ref{corr:rho-rigidity}.

This concludes the proof if $e=1$. If $e > 2$ the same proof works. The only
problem might be that we have torsion prime to the characteristic, but we
excluded this possibility by inverting $e$.
\end{proof}

\paragraph{Remark.} We appeal to \cite{levine2015witt} in order to know that
$\ul \pi_i(\tunit)_0 \otimes \QQ = 0$ for $i>0$. This can also be deduced from
Proposition \ref{prop:app1-ident}, using that $\SH(k)_\QQ^+ = \DM(k,\QQ)$.

There is another (older) notion of rigidity first considered by Suslin
\cite{suslin1983thek}. This corresponds to (1) in the next result. It is a
slightly silly property in our situation, but (2) is a replacement in spirit.
It is related to important results in semialgebraic topology due to Coste-Roy, Delfs
\cite[see in particular Corollary II.6.2]{delfs1991homology} and
Scheiderer \cite{real-and-etale-cohomology}.

\begin{prop}
Let $E \in \SH(k)[\rho^{-1}]$ and $i \in \ZZ$.
\begin{enumerate}[(1)]
\item If $\bar{L}/\bar{K}$ is an extension of algebraically closed fields
  over $k$, then
  \[ \ul{\pi}_i(E)(\bar{K}) = \ul{\pi}_i(E)(\bar{L}) = 0. \]
\item If $L^r/K^r$ is an extension of real closed fields over $k$, then  also
  \[ \ul{\pi}_i(E)(K^r) = \ul{\pi}_i(E)(L^r). \]
\end{enumerate}
\end{prop}
\begin{proof}
As before, by Theorem \ref{thm:final-comparison} and Corollary \ref{corr:Le-exact-ff} we
know that all the homotopy sheaves of $E$ are of the form $eF$, with $F \in
Shv(Spec(k)_\ret)$. For such sheaves we have $eF(\bar{K}) = 0 = eF(\bar{L})$, so
(1) holds.
Since pullback $Spec(K^r)_\ret \to Spec(L^r)_\ret$ induces an isomorphism of
sites, (2) also follows immediately. (See also the first paragraph of the proof
of Lemma \ref{lemm:loc-const-rigid-ret}.)
\end{proof}

\bibliographystyle{amsalphac}
\bibliography{bibliography}

\end{document}

%% file: preamble.tex
\usepackage{amsmath}
\usepackage{amssymb}
\usepackage{amsthm}
\usepackage{amscd}
\usepackage{enumerate}
\usepackage[colorlinks=true]{hyperref}
\usepackage{bbm}
\usepackage{etoolbox}

\usepackage[utf8]{inputenc}

\ifdefvoid{\longdocument}{
\newtheorem*{defn}{Definition}
\newtheorem{prop}{Proposition}
}{
\newtheorem{prop}{Proposition}[chapter]

}
\newtheorem{corr}[prop]{Corollary}
\newtheorem{lemm}[prop]{Lemma}
\newtheorem{thm}[prop]{Theorem}

\newtheorem*{thm*}{Theorem}
\newtheorem*{corr*}{Corollary}
\newtheorem*{prop*}{Proposition}
\newtheorem*{lemm*}{Lemma}
\theoremstyle{definition}

\theoremstyle{remark}

\newcommand{\Aone}{{\mathbb{A}^1}}
\newcommand{\Gm}{{\mathbb{G}_m}}
\newcommand{\Gmp}[1]{{\mathbb{G}_m^{\wedge #1}}}
\newcommand{\id}{\operatorname{id}}
\newcommand{\ZZ}{\mathbb{Z}}
\newcommand{\CC}{\mathbb{C}}

\newcommand{\QQ}{\mathbb{Q}}
\newcommand{\FF}{\mathbb{F}}

\newcommand{\DM}{\mathbf{DM}}

\newcommand{\SH}{\mathbf{SH}}

\newcommand{\iHom}{\ul{\operatorname{Hom}}}

\newcommand{\colim}{\operatorname{colim}}
\newcommand{\hocolim}{\operatorname{hocolim}}

\newcommand{\RR}{\mathbb{R}}

\newcommand{\Mod}{\textbf{-Mod}}

\DeclareRobustCommand{\ul}{\underline}
\newcommand{\heart}{\heartsuit}

\newcommand{\tunit}{\mathbbm{1}}
\newcommand{\ret}{{r\acute{e}t}}

\newcommand{\iso}{\cong}
\newcommand{\wequi}{\simeq}

\newcommand{\Map}{Map}